\documentclass{amsart}
\usepackage{enumitem}
\usepackage{amsmath,amssymb,amsfonts}
\usepackage[all]{xy}
\usepackage{tikz-cd}
\usepackage[total={6.5in,8.75in},
top=1.2in, left=0.9in, includefoot]{geometry}

\renewcommand{\S}{\mathcal{S}}
\newcommand{\MS}{\mathcal{MS}}
\newcommand{\Sh}{\mathcal{S}h}

\newcommand{\B}{\mathcal{B}}

\newcommand{\F}{\mathbb{F}}

\newcommand{\chains}{\mathrm{C}_\bullet}
\newcommand{\cochains}{\mathrm{C}^\bullet}
\newcommand{\tensor}{\otimes}
\newcommand{\Hom}{\mathrm{Hom}}
\newcommand{\End}{\mathrm{End}}
\newcommand{\Mod}{\mathcal{M}od_{k}}

\newtheorem{theorem}{Theorem}

\newtheorem{lemma}[theorem]{Lemma}
\newtheorem{corollary}[theorem]{Corollary}
\theoremstyle{definition}
\newtheorem{definition}[theorem]{Definition}
\newtheorem{example}[theorem]{Example}
\newtheorem{remark}[theorem]{Remark}
\newtheorem{notation}[theorem]{Notation}

\makeatletter
\newcommand*{\defeq}{\mathrel{\rlap{%
			\raisebox{0.3ex}{$\m@th\cdot$}}%
		\raisebox{-0.3ex}{$\m@th\cdot$}}%
	=}
\makeatother

\usetikzlibrary{arrows,decorations.markings}
\tikzset{myptr/.style={decoration={markings,mark=at position 1 with %
			{\arrow[scale=1.5,>=stealth]{>}}},postaction={decorate}}}

\newsavebox\preproduct
\begin{lrbox}{\preproduct}
	\begin{tikzpicture}[scale=.3]
	\draw (0,0)--(0,-.8);
	\draw (0,0)--(.5,.5);
	\draw (0,0)--(-.5,.5);
	\end{tikzpicture} 
\end{lrbox}
\newcommand{\product}{
	\usebox\preproduct}

\newsavebox\precoproduct
	\begin{lrbox}{\precoproduct}
		\begin{tikzpicture}[scale=.3]
		\draw (0,0)--(0,.8);
		\draw (0,0)--(.5,-.5);
		\draw (0,0)--(-.5,-.5);
		\end{tikzpicture}
	\end{lrbox}
\newcommand{\coproduct}{
	\usebox\precoproduct}

\newsavebox\preboundary
\begin{lrbox}{\preboundary}
	\begin{tikzpicture}[scale=.3]
	\draw (0,0)--(0,1.3);
	\draw (.5,0)--(.5,1.3);
	\draw [fill] (0,0) circle [radius=0.1];
	\draw (.9,.6)--(1.3,.6);
	\draw (1.7,0)--(1.7,1.3);
	\draw (2.2,0)--(2.2,1.3);
	\draw [fill] (2.2,0) circle [radius=0.1];
	\end{tikzpicture}
\end{lrbox}
\newcommand{\boundary}{
	\usebox\preboundary}

\newsavebox\precoboundary
\begin{lrbox}{\precoboundary}
	\begin{tikzpicture}[scale=.3]
	\draw (0,0)--(0,1.3);
	\draw (.5,0)--(.5,1.3);
	\draw [fill] (0,1.3) circle [radius=0.1];
	\draw (.9,.6)--(1.3,.6);
	\draw (1.7,0)--(1.7,1.3);
	\draw (2.2,0)--(2.2,1.3);
	\draw [fill] (2.2,1.3) circle [radius=0.1];
	\end{tikzpicture}
\end{lrbox}

\newsavebox\precounit
\begin{lrbox}{\precounit}
	\begin{tikzpicture}[scale=.3]
	\draw (0,0)--(0,1.3);
	\draw [fill] (0,0) circle [radius=0.1];
	\end{tikzpicture}
\end{lrbox}
\newcommand{\counit}{
	\usebox\precounit}

\newsavebox\preidentity
\begin{lrbox}{\preidentity}
	\begin{tikzpicture}[scale=.3]
	\draw (0,0)--(0,1.3);
	\end{tikzpicture}
\end{lrbox}

\newsavebox\preunit
\begin{lrbox}{\preunit}
	\begin{tikzpicture}[scale=.3]
	\draw (0,0)--(0,1.3);
	\draw [fill] (0,1.3) circle [radius=0.1];
	\end{tikzpicture}
\end{lrbox}

\newsavebox\preassociativity
\begin{lrbox}{\preassociativity}
	\begin{tikzpicture}[scale=.2]
	\path[draw] (0,0)--(0,1)--(-1,2);
	\draw (0,1)--(1,2);
	\draw (-.5,1.5)--(0,2);
	\draw (1.2,1)--(1.8,1);
	\path[draw] (3,0)--(3,1)--(2,2);
	\draw (3,1)--(4,2);
	\draw (3.5,1.5)--(3,2);	
	\end{tikzpicture}
\end{lrbox}

\newsavebox\precoassociativity
\begin{lrbox}{\precoassociativity}
	\begin{tikzpicture}[scale=.2]
	\path[draw] (0,0)--(0,-1)--(-1,-2);
	\draw (0,-1)--(1,-2);
	\draw (-.5,-1.5)--(0,-2);
	\draw (1.2,-1)--(1.8,-1);
	\path[draw] (3,0)--(3,-1)--(2,-2);
	\draw (3,-1)--(4,-2);
	\draw (3.5,-1.5)--(3,-2);
	\end{tikzpicture}
\end{lrbox}

\newsavebox\preinvolution
\begin{lrbox}{\preinvolution}
	\begin{tikzpicture}[scale=.2]
	\path[draw] (0,0)--(0,.5)--(-.5,1)--(0,1.5)--(0,2);
	\path[draw] (0,.5)--(.5,1)--(0,1.5);
	\end{tikzpicture}
\end{lrbox}

\newsavebox\preleftcounitality
\begin{lrbox}{\preleftcounitality}
	\begin{tikzpicture}[scale=.3]
	\draw (0,0)--(0,.8);
	\draw (0,0)--(.5,-.5);
	\draw (0,0)--(-.5,-.5);
	\draw [fill] (-.5,-.5) circle [radius=0.1];
	\draw (.7,0)--(1.1,0);
	\path[draw] (1.5,-.5)--(1.5,.8);
	\end{tikzpicture}
\end{lrbox}
\newcommand{\leftcounitality}{
	\usebox\preleftcounitality}

\newsavebox\prerightcounitality
\begin{lrbox}{\prerightcounitality}
	\begin{tikzpicture}[scale=.3]
	\draw (0,0)--(0,.8);
	\draw (0,0)--(.5,-.5);
	\draw (0,0)--(-.5,-.5);
	\draw [fill] (.5,-.5) circle [radius=0.1];
	\draw (-.7,0)--(-1.1,0);
	\path[draw] (-1.5,-.5)--(-1.5,.8);
	\end{tikzpicture}
\end{lrbox}
\newcommand{\rightcounitality}{
	\usebox\prerightcounitality}

\newsavebox\preleftunitality
\begin{lrbox}{\preleftunitality}
	\begin{tikzpicture}[scale=.3]
	\draw (0,0)--(0,-.8);
	\draw (0,0)--(-.5,.5);
	\draw (0,0)--(.5,.5);
	\draw [fill] (.5,.5) circle [radius=0.1];
	\draw (.7,0)--(1.1,0);
	\path[draw] (1.5,.5)--(1.5,-.8);
	\end{tikzpicture}
\end{lrbox}

\newsavebox\prerightunitality
\begin{lrbox}{\prerightunitality}
	\begin{tikzpicture}[scale=.3]
	\draw (0,0)--(0,-.8);
	\draw (0,0)--(-.5,.5);
	\draw (0,0)--(.5,.5);
	\draw [fill] (-.5,.5) circle [radius=0.1];
	\draw (-.7,0)--(-1.1,0);
	\path[draw] (-1.5,.5)--(-1.5,-.8);
	\end{tikzpicture}
\end{lrbox}

\newsavebox\preproductcounit
\begin{lrbox}{\preproductcounit}
	\begin{tikzpicture}[scale=.3]
	\draw (0,0)--(0,-.8);
	\draw (0,0)--(.5,.5);
	\draw (0,0)--(-.5,.5);
	\draw [fill] (0,-.8) circle [radius=0.1];
	\end{tikzpicture}
\end{lrbox}
\newcommand{\productcounit}{
	\usebox\preproductcounit}

\newsavebox\preunitcoproduct
\begin{lrbox}{\preunitcoproduct}
	\begin{tikzpicture}[scale=.3]
	\draw (0,0)--(0,.8);
	\draw (0,0)--(.5,-.5);
	\draw (0,0)--(-.5,-.5);
	\draw [fill] (0,.8) circle [radius=0.1];
	\end{tikzpicture}
\end{lrbox}

\newsavebox\preleibniz
\begin{lrbox}{\preleibniz}
	\begin{tikzpicture}[scale=.245]
	\draw (0,.3)--(0,-.3);
	\draw (0,.3)--(.5,.8);
	\draw (0,.3)--(-.5,.8);
	\draw (0,-.3)--(0,.3);
	\draw (0,-.3)--(.5,-.8);
	\draw (0,-.3)--(-.5,-.8);
	
	\draw (.7,0)--(1.1,0);
	\draw (2.1,.8)--(2.1,.3)--(1.7,0)--(1.7,-.8);
	\draw (2.1,.3)--(2.9,-.3);
	\draw (3.3,.8)--(3.3,0)--(2.9,-.3)--(2.9,-.8);
	
	\draw (3.8,0)--(4.1,0);
	\draw (4.6,.8)--(4.6,0)--(5,-.3)--(5,-.8);
	\draw (5,-.3)--(5.8,.3);
	\draw (5.8,.8)--(5.8,.3)--(6.2,0)--(6.2,-.8);	
	\end{tikzpicture}
\end{lrbox}

\newsavebox\prebialgebra
\begin{lrbox}{\prebialgebra}
	\begin{tikzpicture}[scale=.5]
	\draw (0,.3)--(0,-.3);
	\draw (0,.3)--(.5,.8);
	\draw (0,.3)--(-.5,.8);
	\draw (0,-.3)--(0,.3);
	\draw (0,-.3)--(.5,-.8);
	\draw (0,-.3)--(-.5,-.8);
	
	\draw (.8,0)--(1.2,0);
	
	\draw (2.1,.8)--(2.1,.3)--(1.7,0)--(2.1,-.3)--(2.1,-.8);
	
	\draw (3.1,.8)--(3.1,.3)--(3.5,0)--(3.1,-.3)--(3.1,-.8);

	\draw (2.1,.3)--(3.1,-.3);
	\draw (2.1,-.3)--(2.5,-.06);
	\draw (3.1,.3)--(2.7,.06);	
	\end{tikzpicture}
\end{lrbox}

\newsavebox\precommutativity
\begin{lrbox}{\precommutativity}
	\begin{tikzpicture}[scale=.27]
	\draw (.3,0)--(.3,-1);
	\draw (.3,0)--(.8,.5);
	\draw (.3,0)--(-.2,.5);
	
	\draw (1.2,-.4)--(1.8,-.4);
	
	\draw (3,-.3)--(3,-1);
	\draw (2.5,0)--(3,-.3);
	\draw (3.5,0)--(3,-.3);
	\draw (2.46,0)--(2.9,.18);
	\draw (3.1,.3)--(3.5,.5);
	\draw (3.5,0)--(2.5,.5);
		\end{tikzpicture} 
	\end{lrbox}

\begin{document}
	
	\title{A finitely presented $E_\infty$-prop I: algebraic context}
	\author[A. M. Medina-Mardones]{Anibal M. Medina-Mardones}
	\address{Department of Mathematics, University of Notre Dame}
	\email{amedinam@nd.edu}
	
	\begin{abstract}
		We introduce a finitely presented prop $\mathcal{S} = \{\mathcal{S}(n,m)\}$ in the category of differential graded modules whose associated operad $U(\mathcal{S})=\{\mathcal{S}(1,m)\}$ is a model for the $E_\infty$-operad. This finite presentation allows us to describe a natural $E_\infty$-coalgebra structure on the chains of any simplicial set in terms of only three maps: the Alexander-Whitney diagonal, the augmentation map, and an algebraic version of the join of simplices. The first appendix connects our construction to the Surjection operad of McClure-Smith and Berger-Fresse. The second establishes a duality between the join and AW maps for augmented and non-augmented simplicial sets. A follow up paper \cite{medina2018cellular} constructs a prop corresponding to $\S$ in the category of $CW$-complexes.
	\end{abstract}
	
	\maketitle
	
	\tableofcontents
	
	\section{Introduction}
	A purposeful construction of model for the $E_\infty$-operad is central in most contexts where commutativity up to coherent homotopies plays a role. In this paper, we work in the differential graded context and the property that we are interested in is the size of a presentation in terms of generators and relations. No finitely presented model of the $E_\infty$-operad can exist, but passing to a more general setting with multiple inputs and outputs allows us to finitely present a prop whose associated operad is a model for the $E_\infty$-operad. 
	
	We use this finitely presented prop to describe an $E_\infty$-structure on the chains of any simplicial set induced by only three operations: the Alexander-Whitney diagonal, the augmentation map, and an algebraic version of the join of simplices. Additionally, we describe the Surjection operad of McClure-Smith \cite{mcclure2003multivariable} and Berger-Fresse \cite{berger2004combinatorial} as the operad associated to a finitely presented prop $\MS$ with the same generators as $\S$ and five more relations. This description recovers the $E_\infty$-structure defined by these authors on the chains of simplicial sets as that induced by $\S$. 
	
	This is the first part of two papers. The second \cite{medina2018cellular} deals with a CW realization of the combinatorial results here. The operad associated to the prop corresponding to $\MS$ will be shown to be isormorphic to an Arc Operad \cite{kaufmann03arc, kaufmann09dimension}. This and the relation to the surjection operad links the current paper to Deligne's conjecture in the form \cite{mcclure2003multivariable, berger2004combinatorial,kaufmann07spinless}.	The combinatorial formulation, we present here is of independent interest providing a concise graphical context. Furthermore, we note that the coproduct we use corresponds, in a string topology interpretation \cite{tradler07string, kaufmann08frobenious}, to a degree 0 coproduct \cite{kaufmann2018detailed}. This is to be distinguished from the degree 1 coproduct that is more widely known as the Goresky-Hingston product \cite{goresky09loop, kaufmann2018detailed}. Here we show that the degree 0 coproduct has its own importance as it is the key ingredient to realizing the $E_\infty$-structure.
	
	The outline of this paper is as follows: in the second section, we review the material on operads and props needed for the rest of the paper. In the third section, we give the finite presentation of the prop $\S$ and compute its homology. In the fourth section, we construct a natural $\S$-bialgebra structure on the chains of standard simplices, based solely on the Alexander-Whitney diagonal, the augmentation map, and a differential graded join map. These $\S$-bialgebra structures induce a natural $U(\S)$-coalgebra structure on the chains of any simplicial set. Here $U(\S)$ is the canonical \mbox{$E_\infty$-operad} with $U(\S)(m)=\S(1,m)$ for all $m\geq 0$. In the first appendix, we realize the Surjection operad of McClure-Smith \cite{mcclure2003multivariable} and Berger-Fresse \cite{berger2004combinatorial} as the operad associated to a finitely presented prop $\MS$ which is a quotient of $\S$. We also describe how the $E_\infty$-structure defined by these authors on chains of simplicial sets is generated by the Alexander-Whitney, augmentation, and join maps. In the second appendix, we construct natural opposite $\S$-bialgebra structure on the chains of standard augmented simplices. 
	
	\subsection*{Acknowledgments} We would like to thank Bruno Vallette, Stephan Stolz, Ralph Cohen, Dennis Sullivan, Dev Sinha, Greg Friedman, Kathryn Hess, and Ralph Kaufmann for their insights, questions, and comments about this project. 
	
	\section{Preliminaries about operads and props}
	
	\subsection*{Conventions.} We will work in the symmetric monoidal category $(\Mod,\tensor,k)$ of differential homologically graded $k$-modules with $k$ a commutative ring. We will use the notation $\overline{n}$ for $\{1,\dots,n\}$ with $\overline{0}$ standing for the empty set.
	
	\subsection{$E_\infty$-operads}
	We say that an operad $\mathcal{O}$ is $\Sigma$\textbf{-free} if $\mathcal{O}(m)$ is a $k[\Sigma_m]$-free module for every $m$. A \mbox{$\Sigma$-\textbf{free resolution}} of an operad $\mathcal{O}$ is a morphism from a $\Sigma$-free operad to $\mathcal{O}$ inducing a homology isomorphism in each arity $m$. 
	
	For any $C\in\Mod$, we have two naturally associated operads given for any $m\geq0$ by $$\End(C)(m) = \Hom(C,C^{\tensor m})\text{\ \ and \ \ }\End^\mathrm{op}(C)(m) = \Hom(C^{\tensor m},C).$$
	For any $C\in\Mod$ there are two types of representations of an operad $\mathcal{O}$ on $C$. They are referred to as $\mathcal{O}$-\textbf{coalgebra} and $\mathcal{O}$-\textbf{algebra} structures and are respectively given by operad morphisms
	$$\mathcal{O}\to\End(C) \text{\ \ and \ }\mathcal{O}\to\End^\mathrm{op}(C).$$
	The operad $\End(k)$ is of particular importance, with $\End(k)$-(co)algebras defining usual (co)commutative, (co)associative and (co)unital (co)algebras. 
	
	Following May, for example \cite{kriz1995operads}, an operad $\mathcal{O}$ is called an $E_\infty$-\textbf{operad} if it is a $\Sigma$-free resolution of $\End(k)$ and $\mathcal{O}(0)=k$.
	
	\subsection{Props} A \textbf{prop} is a strict symmetric monoidal category $\mathcal{P} = (\mathcal{P}, \odot, 0)$ enriched in $\Mod$ generated by a single object. For a prop $\mathcal{P}$ with generator $p$ we denote $Mor_\mathcal{P}(p^{\odot n}, p^{\odot m})$ by $\mathcal{P}(n,m)$ and notice, induced from the symmetry of the monoidal structure, the commuting right and left actions of $\Sigma_n$ and $\Sigma_m$ respectively. Therefore, we think of the data of a prop as a $\Sigma$-\textbf{bimodule}, i.e., a collection $\mathcal{P} = \big\{\mathcal{P}(n,m)\big\}_{n,m\geq0}$ of differential graded $k$-modules with commuting actions of $\Sigma_n$ and $\Sigma_m$ together with three types of maps:
	\begin{gather*}
		\circ_h : \mathcal{P}(n_1,m_1) \tensor \cdots \tensor \mathcal{P}(n_s,m_s) \to \mathcal{P}(n_1+\cdots+n_s,m_1+\cdots+m_s) \\
		\circ_v : \mathcal{P}(r,m) \tensor \mathcal{P}(n,r) \to \mathcal{P}(n,m) \\
		\eta : k \to \mathcal{P}(n,n).
	\end{gather*}
	These types of maps are referred to respectively as horizontal compositions, vertical compositions, and units. They come from the monoidal product, the categorical composition of $\mathcal{P}$, and its identity morphisms. 
	
	For any $C\in\Mod$ we have two naturally associated props given for any $n,m\geq0$ by 
	$$\End(C)(n,m) = \Hom(C^{\tensor n},C^{\tensor m}) \text{\ \ and \ \ }\End^\mathrm{op}(C)(n,m) = \Hom(C^{\tensor m},C^{\tensor n}).$$ 
	There are two types of representations of a prop $\mathcal{P}$ on $C$, they are referred to as $\mathcal{P}$-\textbf{bialgebra} and \textbf{opposite} $\mathcal{P}$-\textbf{bialgebra} structures and are respectively given by prop morphisms 
	$$\mathcal{P}\to\End(C) \text{\ \ and \ }\mathcal{P}\to\End^\mathrm{op}(C).$$
	Let $U$ be the functor from the category of props to that of operads given by naturally inducing from a prop $\mathcal{P}$ an operad structure on the $\Sigma$-module $U(\mathcal{P})=\{\mathcal{P}(1,m)\}_{m\geq0}$. Notice that a $\mathcal{P}$-bialgebra (resp. opposite $\mathcal{P}$-bialgebra) structure on $C$ induces a $U(\mathcal{P})$-coalgebra (resp. $U(\mathcal{P})$-algebra) structure on $C$.
	
	Following Boardman and Vogt \cite{boardman2006homotopy}, a prop $\mathcal{P}$ is called an $E_\infty$-\textbf{prop} if $U(\mathcal{P})$ is an $E_\infty$-operad.
	
	\subsection{Free props}	As described, for example, in \cite{markl2008operads}, the \textbf{free prop} generated by a $\Sigma$-bimodule is constructed using open directed graphs with no directed loops that are enriched with a labeling described next. We think of each directed edge as built from two compatibly directed half-edges. For each vertex $v$ of a directed graph $G$, we have the sets $in(v)$ and $out(v)$ of half-edges that are respectively incoming to and outgoing from $v$. Half-edges that do not belong to $in(v)$ or $out(v)$ for any $v$ are divided into the disjoint sets $in(G)$ and $out(G)$ of incoming and outgoing external half-edges. Recall that $\overline{n} = \{1,\dots,n\}$. Denoting the cardinality of a finite set $S$ by $|S|$, the labeling is given by bijections  
	$$ \overline{|in(G)|}\to in(G)\hspace*{1cm}\overline{|out(G)|}\to out(G)$$
	and
	$$ \overline{|in(v)|}\to in(v)\hspace*{1cm}\overline{|out(v)|}\to out(v)$$ 
	for every vertex $v$. We refer to such labeled directed graphs with no directed loops as $(n,m)$\textbf{-graph}. We consider the right action of $\Sigma_n$ and the left action of $\Sigma_m$ on a $(n,m)$-graph given respectively by permuting the labels of $in(G)$ and $out(G)$. 
	
	A $k$-module basis for the $(n,m)$-part of the $\Sigma$-bimodule underlying the free prop generated by a $\Sigma$-bimodule $\B$ is given by all isomorphism classes of $(n,m)$-graphs which are $\B(n,m)$-decorated in the following way: to every vertex $v$ of one such $G$, one assigns an element $p \in \B(|in(v)|, |out(v)|)$ and introduces the equivalence relations:  
	\begin{center}
		\boxed{\begin{tikzpicture}[scale=.6]
			\node at (0,2.1) {$1$}; \draw[->] (0,1.7) -- (0,1);
			\node at (.8,2.1) {$\dots$};
			\node at (2,2.1) {$\scriptstyle \overline{|in(v)|}$}; \draw[->] (2,1.7) -- (2,1);
			
			\draw (2.4,.5) arc (0:360:1.4cm and .5cm); \node at (1, .48) {$a\,p+q$};
			
			\node at (0,-1.1) {$1$}; \draw[<-](0,-.7) -- (0,0);
			\node at (.8,-1.1) {$\dots$};
			\node at (2,-1.1) {$\scriptstyle \overline{|out(v)|}$}; \draw[<-] (2,-.7) -- (2,0);
			
			\node at (3.6,.5) {$\sim$};  
			\end{tikzpicture}
			\hspace*{.2cm}
			\begin{tikzpicture}[scale=.6]
			\node at (-.8,.5) {$a$};
			\node at (0,2.1) {$1$}; \draw[->] (0,1.7) -- (0,1);
			\node at (.8,2.1) {$\dots$};
			\node at (2,2.1) {$\scriptstyle \overline{|in(v)|}$}; \draw[->] (2,1.7) -- (2,1);
			
			\draw (2.4,.5) arc (0:360:1.4cm and .5cm); \node at (1, .4) {$p$};
			
			\node at (0,-1.1) {$1$}; \draw[<-](0,-.7) -- (0,0);
			\node at (.8,-1.1) {$\dots$};
			\node at (2,-1.1) {$\scriptstyle \overline{|out(v)|}$}; \draw[<-] (2,-.7) -- (2,0);
			\end{tikzpicture}
			\hspace*{-.01cm}
			\begin{tikzpicture}[scale=.6]
			\node at (-.9,.5) {$+$};
			\node at (0,2.1) {$1$}; \draw[->] (0,1.7) -- (0,1);
			\node at (.8,2.1) {$\dots$};
			\node at (2,2.1) {$\scriptstyle \overline{|in(v)|}$}; \draw[->] (2,1.7) -- (2,1);
			
			\draw (2.4,.5) arc (0:360:1.4cm and .5cm); \node at (1, .4) {$q$};
			
			\node at (0,-1.1) {$1$}; \draw[<-](0,-.7) -- (0,0);
			\node at (.8,-1.1) {$\dots$};
			\node at (2,-1.1) {$\scriptstyle \overline{|out(v)|}$}; \draw[<-] (2,-.7) -- (2,0);
			\end{tikzpicture}
		}
		
		\boxed{\begin{tikzpicture}[scale=.6]
			\node at (0,2.1) {$1$}; \draw[->] (0,1.7) -- (0,1);
			\node at (.8,2.1) {$\dots$};
			\node at (2,2.1) {$\scriptstyle \overline{|in(v)|}$}; \draw[->] (2,1.7) -- (2,1);
			
			\draw (2.4,.5) arc (0:360:1.4cm and .5cm); \node at (1, .45) {$\tau^{-1} p\,\sigma^{-1}$};
			
			\node at (0,-1.1) {$1$}; \draw[<-](0,-.7) -- (0,0);
			\node at (.8,-1.1) {$\dots$};
			\node at (2,-1.1) {$\scriptstyle \overline{|out(v)|}$}; \draw[<-] (2,-.7) -- (2,0);
			
			\node at (3.5,.5) {$\sim$}; 
			\end{tikzpicture}\; 
			\begin{tikzpicture}[scale=.6]
			\node at (0,2.1) {$\scriptstyle \sigma(1)$}; \draw[->] (0,1.7) -- (0,1);
			\node at (1,2.1) {$\dots$};
			\node at (2.4,2.1) {$\scriptstyle \sigma(\overline{|in(v)|})$}; \draw[->] (2,1.7) -- (2,1);
			
			\draw (2.4,.5) arc (0:360:1.4cm and .5cm); \node at (1, .45) {$p$};
			
			\node at (0,-1.1) {$\scriptstyle \tau(1)$}; \draw[<-](0,-.7) -- (0,0);
			\node at (1,-1.1) {$\dots$};
			\node at (2.5,-1.1) {$\scriptstyle \tau(\overline{|out(v)|})$}; \draw[<-] (2,-.7) -- (2,0);
			\end{tikzpicture}
		} \qquad
		\boxed{\begin{tikzpicture}[scale=.6]
			\node at (0,2.1) {$1$}; \draw[->] (0,1.7) -- (0,1);
			\node at (.8,2.1) {$\dots$};
			\node at (2,2.1) {$\scriptstyle \overline{|in(v)|}$}; \draw[->] (2,1.7) -- (2,1);
			
			\draw (2.4,.5) arc (0:360:1.4cm and .5cm); \node at (1, .45) {$\eta(1)$};
			
			\node at (0,-1.1) {$1$}; \draw[<-](0,-.7) -- (0,0);
			\node at (.8,-1.1) {$\dots$};
			\node at (2,-1.1) {$\scriptstyle \overline{|out(v)|}$}; \draw[<-] (2,-.7) -- (2,0);
			
			\node at (3.5,.5) {$\sim$}; 
			\end{tikzpicture}\; 
			\begin{tikzpicture}[scale=.6]
			\node at (0,2.1) {$1$}; 
			\draw[->] (0,1.7) -- (0,-.7);
			\node at (.8,2.1) {$\dots$};
			\node at (2,2.1) {$\scriptstyle \overline{|in(v)|}$};
			\draw[->] (2,1.7) -- (2,-.7);
			
			\node at (0,-1.1) {$1$}; 
			\node at (.8,-1.1) {$\dots$};
			\node at (2,-1.1) {$\scriptstyle \overline{|out(v)|}$}; 
			\end{tikzpicture}
		}
	\end{center}
	where $p,q\in \B(n,m)$, $\sigma\in\Sigma_n$, $\tau\in\Sigma_m$ and $a\in k$. Therefore, we have a description of the form $$F(\B)(n,m)=\left(\bigoplus_{G}\bigotimes_{v}\B\big(|in(v)|,|out(v)|\big)\right)\bigg/\sim$$
	that endows $F(\B)(n,m)$ the structure of a differential grade $k$-module. The vertical compositions are given by grafting, whereas the horizontal composition is given by unions of graphs with the corresponding relabeling of external half-edges.  
	
	\subsection{Presentations} We will describe what we mean by a \textbf{presentation} $(G,\partial, R)$ of a prop. The first piece of data is a collection $G=\{G(n,m)_d\}$ of sets thought of as generators of biarity $(n,m)$ and homological degree $d$. Consider the free $\Sigma$-bimodule over the category of graded $k$-modules having basis $G$, i.e., in biarity $(n,m)$ the $k[\Sigma_n^{\mathrm{op}}\times\Sigma_m]$~-~module in homological degree $d$ is free with basis $G(n,m)_d$. The free prop over the category of graded $k$-modules generated by this $\Sigma$-bimodule is denoted $F(G)$. The second piece of data is a map $\partial: G\to F(G)$ of degree $-1$, thought of as the boundary of the generators. Extending this map as a derivation we promote $F(G)$ to a prop over $\Mod$. The third piece of data is a collection $R=\{R(n,m)_d\}$ of subsets of $F(G)(n,m)$ thought of as relations. Denote by $\langle R\rangle$ the (differential graded) prop ideal generated by $R$ in $F(G)$. We say that the triple $(G,\partial,R)$ is a presentation of the prop $F(G)/\langle R\rangle$.
	
	\subsection{Immersion convention} If a graph immersed in the plane appears representing an $(n,m)$-graph, the convention we  follow is to give the direction from top to bottom and the labeling from left to right. For example:
	
	\begin{center}
		\boxed
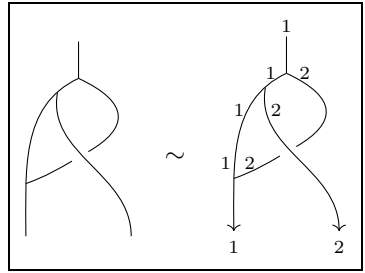
{\begin{tikzpicture}[scale=.7]
			\draw (1,3.7) to (1,3); 
			
			\draw (1,3) to [out=205, in=90] (0,0);
			
			\draw [shorten >= 0cm] (.6,2.73) to [out=-100, in=90] (2,0);
			
			\draw [shorten >= .15cm] (1,3) to [out=-25, in=30, distance=1.1cm] (1,1.5);
			\draw [shorten <= .1cm] (1,1.5) to [out=210, in=20] (0,1);
			
			\node at (1,3.9){};
			\node at (0,-.32){};
			\node at (2,-.32){};
			
			\node at (3,1.5){$\sim$\ \ \ };
			\end{tikzpicture}
			\begin{tikzpicture}[scale=.7]
			\draw (1,3.7) to (1,3); 
			
			\draw [->] (1,3) to [out=205, in=90] (0,0);
			
			\draw [shorten >= 0cm,->] (.6,2.73) to [out=-100, in=90] (2,0);
			
			\draw [shorten >= .15cm] (1,3) to [out=-25, in=30, distance=1.1cm] (1,1.5);
			\draw [shorten <= .1cm] (1,1.5) to [out=210, in=20] (0,1);
			
			\node at (1,3.9){$\scriptstyle 1$};
			
			\node at (.7,3){$\scriptstyle 1$};
			\node at (1.35,3){$\scriptstyle 2$};
			
			\node at (.1,2.3){$\scriptstyle 1$};
			\node at (.8,2.3){$\scriptstyle 2$};
			
			\node at (-.15,1.3){$\scriptstyle 1$};
			\node at (.3,1.3){$\scriptstyle 2$};
			
			\node at (0,-.3){$\scriptstyle 1$};
			\node at (2,-.3){$\scriptstyle 2$};
			\end{tikzpicture}}
	\end{center}
	
	\section{The prop $\S$ }
	In this section we define the central object of this note: the prop $\S$. We give a finite presentation of $\S$ and show it is an $E_\infty$-prop.
	
	\subsection{The definition of $\S$}
	
	\begin{definition}\label{Prop S}
		Let $\S$ be the prop generated by 
		$$\counit\in\S(1,0)_0\hspace*{.6cm}\coproduct\in\S(1,2)_0\hspace*{.6cm}\product\in\S(2,1)_1$$ 
		with differential $$\partial\ \counit=0\hspace*{.6cm}\partial\ \coproduct=0\hspace*{.6cm}\partial\ \product=\ \boundary$$
		and restricted by the relations $$\productcounit\hspace*{.6cm}\leftcounitality\hspace*{.6cm}\rightcounitality\,.$$ 
	\end{definition}
	
	\subsection{The homology of $\S$}	
	\begin{lemma}
		Let 
		\begin{equation*}	
			\overline{\End}(k)(n,m)=
			\begin{cases} 
				\End(k)(n,m) & \text{ if } n>0 \\
				\hspace*{1cm}0 & \text{ if } n=0.
			\end{cases}
		\end{equation*}
		The map $\S\to\overline{\End}(k)$ defined on generators by 
		$$\ \counit\ \mapsto(1\mapsto1)\in\Hom(k,k^{\tensor0})$$ 
		$$\coproduct\mapsto(1\mapsto 1\tensor1)\in\Hom(k,k^{\tensor2})$$ 
		$$\product\mapsto0\in\Hom(k^{\tensor2},k)$$
		induces a homology isomorphism in each biarity.
	\end{lemma}

	\begin{proof} 			
		For $n=0$, $\S(0,m)=0=\overline{\End}(k)(0,m)$. For $n>0$ and $m\geq0$ we start by showing the complexes $\S(n,m)$ and $\S(n,m+1)$ are chain homotopy equivalent. Consider the collection of maps $\{i:\S(n,m)\to\S(n,m+1)\}$ described by the following diagram\\
		\begin{equation*}
		\boxed{\begin{tikzpicture}[scale=.6]
			\node at (0,2) {$1$}; \draw (0,1.7) -- (0,1);
			\node at (1,2) {$\dots$};
			\node at (2,2) {$n$}; \draw (2,1.7) -- (2,.9);
			
			\draw [dashed] (0,0) rectangle (2,1); \node at (1, .5) {$G$};
			
			\node at (0,-1) {$1$}; \draw (0,-.7) -- (0,0);
			\node at (1,-1) {$\dots$};
			\node at (2,-1) {$m$}; \draw (2,-.7) -- (2,0);
			
			\node at (3,.75) {$i$}; \draw [|->] (2.5,.5) -- (3.5,.5);
			\end{tikzpicture}
			\begin{tikzpicture}[scale=.6]
			\node at (0,2) {$1$}; \draw (0,1.7) -- (0,1);
			\node at (1,2) {$\dots$};
			\node at (2,2) {$n$}; \draw (2,1.7) -- (2,.9);
			
			\draw [dashed] (0,0) rectangle (2,1); \node at (1, .5) {$G$};
			
			\node at (-.5,-1) {$1$}; \draw (-.5,-.7)  to[out=90,in=-140]  (0,1.4);
			\node at (0,-1) {$2$}; \draw (0,-.7) -- (0,0);
			\node at (.75,-1) {$\dots$};
			\node at (2,-1) {$m+1$}; \draw (2,-.7) -- (2,0);
			\end{tikzpicture}}
		\end{equation*}\\
		Consider also the collection of maps $\{r:\S(n,m+1)\to\S(n,m)\}$ described by\\
		\begin{equation*}
		\boxed{\begin{tikzpicture}[scale=.6]
			\node at (0,2) {$1$}; \draw (0,1.7) -- (0,1);
			\node at (1,2) {$\dots$};
			\node at (2,2) {$n$}; \draw (2,1.7) -- (2,.9);
			
			\draw [dashed] (0,0) rectangle (2,1); \node at (1, .5) {$G$};
			
			\draw (0,-.7) -- (0,0);
			\node at (0,-1) {$1$};
			\node at (.5,-1) {$2$}; \draw (.5,-.7) -- (.5,0);
			\node at (1,-1) {$\cdot\!\cdot\!\cdot$};
			\node at (2.1,-1) {$m+1$}; \draw (2,-.7) -- (2,0);
			
			\node at (3,.75) {$r$}; \draw [|->] (2.5,.5) -- (3.5,.5);
			\end{tikzpicture}
			\ 
			\begin{tikzpicture}[scale=.6]
			\node at (0,2) {$1$}; \draw (0,1.7) -- (0,1);
			\node at (1,2) {$\dots$};
			\node at (2,2) {$n$}; \draw (2,1.7) -- (2,.9);
			
			\draw [dashed] (0,0) rectangle (2,1); \node at (1, .5) {$G$};
			
			\node at(0,-.7) {$\bullet$}; \draw (0,-.7) -- (0,0);
			\node at (.5,-1) {$1$}; \draw (.5,-.7) -- (.5,0);
			\node at (1.25,-1) {$\dots$};
			\node at (2,-1) {$m$}; \draw (2,-.7) -- (2,0);
			\end{tikzpicture}}
		\end{equation*}\\
		The diagram below shows that $r\circ i$ is the identity\\
		\begin{equation*}
		\boxed{\begin{tikzpicture}[scale=.6]
			\node at (0,2) {$1$}; \draw (0,1.7) -- (0,1);
			\node at (1,2) {$\dots$};
			\node at (2,2) {$n$}; \draw (2,1.7) -- (2,.9);
			
			\draw [dashed] (0,0) rectangle (2,1); \node at (1, .5) {$G$};
			
			\node at (0,-1) {$1$}; \draw (0,-.7) -- (0,0);
			\node at (1,-1) {$\dots$};
			\node at (2,-1) {$m$}; \draw (2,-.7) -- (2,0);
			
			\node at (3,.75) {$i$}; \draw [|->] (2.5,.5) -- (3.5,.5);
			\end{tikzpicture}
			\begin{tikzpicture}[scale=.6]
			\node at (0,2) {$1$}; \draw (0,1.7) -- (0,1);
			\node at (1,2) {$\dots$};
			\node at (2,2) {$n$}; \draw (2,1.7) -- (2,.9);
			
			\draw [dashed] (0,0) rectangle (2,1); \node at (1, .5) {$G$};
			
			\node at (-.5,-1) {$1$}; \draw (-.5,-.7)  to[out=90,in=-140]  (0,1.4);
			\node at (0,-1) {$2$}; \draw (0,-.7) -- (0,0);
			\node at (.75,-1) {$\dots$};
			\node at (2,-1) {$m+1$}; \draw (2,-.7) -- (2,0);
			
			\node at (3,.75) {$r$}; \draw [|->] (2.5,.5) -- (3.5,.5);
			\end{tikzpicture}
			\begin{tikzpicture}[scale=.6]
			\node at (0,2) {$1$}; \draw (0,1.7) -- (0,1);
			\node at (1,2) {$\dots$};
			\node at (2,2) {$n$}; \draw (2,1.7) -- (2,.9);
			
			\draw [dashed] (0,0) rectangle (2,1); \node at (1, .5) {$G$};
			
			\node at (-.5,-.7) {$\bullet$}; \draw (-.5,-.7)  to[out=90,in=-140]  (0,1.4);
			\node at (0,-1) {$1$}; \draw (0,-.7) -- (0,0);
			\node at (1,-1) {$\dots$};
			\node at (2,-1) {$m$}; \draw (2,-.7) -- (2,0);
			\node at (2.7,.5) {$=$}; 
			\end{tikzpicture}
			\! 
			\begin{tikzpicture}[scale=.6]
			\node at (0,2) {$1$}; \draw (0,1.7) -- (0,1);
			\node at (1,2) {$\dots$};
			\node at (2,2) {$n$}; \draw (2,1.7) -- (2,.9);
			
			\draw [dashed] (0,0) rectangle (2,1); \node at (1, .5) {$G$};
			
			\node at (0,-1) {$1$}; \draw (0,-.7) -- (0,0);
			\node at (1,-1) {$\dots$};
			\node at (2,-1) {$m$}; \draw (2,-.7) -- (2,0);
			\end{tikzpicture}}
		\end{equation*}\\
		Let us compute diagramatically the composition $i\circ r$\\
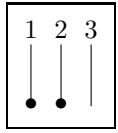
		\begin{equation*}
		\boxed{\begin{tikzpicture}[scale=.6]
			\node at (0,2) {$1$}; \draw (0,1.7) -- (0,1);
			\node at (1,2) {$\dots$};
			\node at (2,2) {$n$}; \draw (2,1.7) -- (2,.9);
			
			\draw [dashed] (0,0) rectangle (2,1); \node at (1, .5) {$G$};
			
			\draw (0,-.7) -- (0,0);
			\node at (0,-1) {$1$};
			\node at (.5,-1) {$2$}; \draw (.5,-.7) -- (.5,0);
			\node at (1.25,-1) {$\dots$};
			\node at (2,-1) {$m$}; \draw (2,-.7) -- (2,0);
			
			\node at (3,.75) {$r$}; \draw [|->] (2.5,.5) -- (3.5,.5);
			\end{tikzpicture}
			\   
			\begin{tikzpicture}[scale=.6]
			\node at (0,2) {$1$}; \draw (0,1.7) -- (0,1);
			\node at (1,2) {$\dots$};
			\node at (2,2) {$n$}; \draw (2,1.7) -- (2,.9);
			
			\draw [dashed] (0,0) rectangle (2,1); \node at (1, .5) {$G$};
			
			\node at(0,-.7) {$\bullet$}; \draw (0,-.7) -- (0,0);
			\node at (.5,-1) {$1$}; \draw (.5,-.7) -- (.5,0);
			\node at (.9,-1) {$\cdot\!\cdot\!\cdot$};
			\node at (2,-1) {$m-1$}; \draw (2,-.7) -- (2,0);
			
			\node at (3,.75) {$i$}; \draw [|->] (2.5,.5) -- (3.5,.5);
			\end{tikzpicture}
			\begin{tikzpicture}[scale=.6]
			\node at (0,2) {$1$}; \draw (0,1.7) -- (0,1);
			\node at (1,2) {$\dots$};
			\node at (2,2) {$n$}; \draw (2,1.7) -- (2,.9);
			
			\draw [dashed] (0,0) rectangle (2,1); \node at (1, .5) {$G$};
			
			\node at (-.5,-1) {$1$}; \draw (-.5,-.7)  to[out=90,in=-140]  (0,1.4);
			\node at (0,-.7) {$\bullet$}; \draw (0,-.7) -- (0,0);
			\node at (.5,-1) {$2$}; \draw (.5,-.7) -- (.5,0);
			\node at (1.25,-1) {$\dots$};
			\node at (2,-1) {$m$}; \draw (2,-.7) -- (2,0);
			\end{tikzpicture}}
		\end{equation*}\\
		Define the collection of homotopies $\{H:\S(n,m)_\bullet\to\S(n,m)_{\bullet+1}\}$ by\\
		\begin{equation*}
		\boxed{\begin{tikzpicture}[scale=.6]
			\node at (0,2) {$1$}; \draw (0,1.7) -- (0,1);
			\node at (1,2) {$\dots$};
			\node at (2,2) {$n$}; \draw (2,1.7) -- (2,.9);
			
			\draw [dashed] (0,0) rectangle (2,1); \node at (1, .5) {$G$};
			
			\node at (0,-1) {$1$}; \draw (0,-.7) -- (0,0);
			\node at (1,-1) {$\dots$};
			\node at (2,-1) {$m$}; \draw (2,-.7) -- (2,0);
			
			\node at (3,.75) {$H$}; \draw [|->] (2.5,.5) -- (3.5,.5);
			\end{tikzpicture}
			\  
			\begin{tikzpicture}[scale=.6]
			\node at (0,2) {$1$}; \draw (0,1.7) -- (0,1);
			\node at (1,2) {$\dots$};
			\node at (2,2) {$n$}; \draw (2,1.7) -- (2,.9);
			
			\draw [dashed] (0,0) rectangle (2,1); \node at (1, .5) {$G$};
			
			\draw (0,-.4)  to[out=140,in=-140]  (0,1.4);
			\node at (0,-1) {$1$}; \draw (0,-.7) -- (0,0);
			\node at (1,-1) {$\dots$};
			\node at (2,-1) {$m$}; \draw (2,-.7) -- (2,0);
			\end{tikzpicture}}	
		\end{equation*}\\
		We now compute the value of $H\circ\partial$ \\
		\begin{equation*}
		\boxed{\begin{tikzpicture}[scale=.6]
			\node at (0,2) {$1$}; \draw (0,1.7) -- (0,1);
			\node at (1,2) {$\dots$};
			\node at (2,2) {$n$}; \draw (2,1.7) -- (2,.9);
			
			\draw [dashed] (0,0) rectangle (2,1); \node at (1, .5) {$G$};
			
			\node at (0,-1) {$1$}; \draw (0,-.7) -- (0,0);
			\node at (1,-1) {$\dots$};
			\node at (2,-1) {$m$}; \draw (2,-.7) -- (2,0);
			
			\node at (3,.75) {$\partial$}; \draw [|->] (2.5,.5) -- (3.5,.5);
			\end{tikzpicture}
			\; 
			\begin{tikzpicture}[scale=.6]
			\node at (0,2) {$1$}; \draw (0,1.7) -- (0,1);
			\node at (1,2) {$\dots$};
			\node at (2,2) {$n$}; \draw (2,1.7) -- (2,.9);
			
			\draw [dashed] (0,0) rectangle (2,1); \node at (1, .5) {$\partial G$};
			
			\node at (0,-1) {$1$}; \draw (0,-.7) -- (0,0);
			\node at (1,-1) {$\dots$};
			\node at (2,-1) {$m$}; \draw (2,-.7) -- (2,0);
			
			\node at (3,.75) {$H$}; \draw [|->] (2.5,.5) -- (3.5,.5);
			\end{tikzpicture}
			\ 
			\begin{tikzpicture}[scale=.6]
			\node at (0,2) {$1$}; \draw (0,1.7) -- (0,1);
			\node at (1,2) {$\dots$};
			\node at (2,2) {$n$}; \draw (2,1.7) -- (2,.9);
			
			\draw [dashed] (0,0) rectangle (2,1); \node at (1, .5) {$\partial G$};
			
			\draw (0,-.4)  to[out=140,in=-140]  (0,1.4);
			\node at (0,-1) {$1$}; \draw (0,-.7) -- (0,0);
			\node at (1,-1) {$\dots$};
			\node at (2,-1) {$m$}; \draw (2,-.7) -- (2,0);
			\end{tikzpicture}}	
		\end{equation*}\\
		And the value of $\partial\circ H$\\
		\begin{equation*}
		\boxed{\begin{tikzpicture}[scale=.6]
			\node at (0,2) {$1$}; \draw (0,1.7) -- (0,1);
			\node at (1,2) {$\dots$};
			\node at (2,2) {$n$}; \draw (2,1.7) -- (2,.9);
			
			\draw [dashed] (0,0) rectangle (2,1); \node at (1, .5) {$G$};
			
			\node at (0,-1) {$1$}; \draw (0,-.7) -- (0,0);
			\node at (1,-1) {$\dots$};
			\node at (2,-1) {$m$}; \draw (2,-.7) -- (2,0);
			
			\node at (3,.8) {$H$}; \draw [|->] (2.5,.5) -- (3.5,.5);
			\end{tikzpicture}
			\  
			\begin{tikzpicture}[scale=.6]
			\node at (0,2) {$1$}; \draw (0,1.7) -- (0,1);
			\node at (1,2) {$\dots$};
			\node at (2,2) {$n$}; \draw (2,1.7) -- (2,.9);
			
			\draw [dashed] (0,0) rectangle (2,1); \node at (1, .5) {$G$};
			
			\draw (0,-.4)  to[out=140,in=-140]  (0,1.4);
			\node at (0,-1) {$1$}; \draw (0,-.7) -- (0,0);
			\node at (1,-1) {$\dots$};
			\node at (2,-1) {$m$}; \draw (2,-.7) -- (2,0);
			
			\node at (3,.8) {$\partial$}; \draw [|->] (2.5,.5) -- (3.5,.5);
			\end{tikzpicture}
			\begin{tikzpicture}[scale=.6]
			\node at (0,2) {$1$}; \draw (0,1.7) -- (0,1);
			\node at (1,2) {$\dots$};
			\node at (2,2) {$n$}; \draw (2,1.7) -- (2,.9);
			
			\draw [dashed] (0,0) rectangle (2,1); \node at (1, .5) {$G$};
			
			\node at (-.5,-.4){$\bullet$}; \draw (-.5,-.4)  to[out=90,in=-140]  (0,1.4);
			\node at (0,-1) {$1$}; \draw (0,-.7) -- (0,0);
			\node at (1,-1) {$\dots$};
			\node at (2,-1) {$m$}; \draw (2,-.7) -- (2,0);
			
			\node at (2.5,.5) {$-$}; 
			\end{tikzpicture}
			\!
			\begin{tikzpicture}[scale=.6]
			\node at (0,2) {$1$}; \draw (0,1.7) -- (0,1);
			\node at (1,2) {$\dots$};
			\node at (2,2) {$n$}; \draw (2,1.7) -- (2,.9);
			
			\draw [dashed] (0,0) rectangle (2,1); \node at (1, .5) {$G$};
			
			\node at (0,-.3){$\bullet$}; \draw (0,-.3) -- (0,0);
			\node at (0,-1) {$1$}; \draw (0,-.7)  to[out=140,in=-140]  (0,1.4);
			\node at (1,-1) {$\dots$};
			\node at (2,-1) {$m$}; \draw (2,-.7) -- (2,0);
			
			\node at (2.5,.5) {$-$}; 
			\end{tikzpicture}		
			\begin{tikzpicture}[scale=.6]
			\node at (0,2) {$1$}; \draw (0,1.7) -- (0,1);
			\node at (1,2) {$\dots$};
			\node at (2,2) {$n$}; \draw (2,1.7) -- (2,.9);
			
			\draw [dashed] (0,0) rectangle (2,1); \node at (1, .5) {$\partial G$};
			
			\draw (0,-.4)  to[out=140,in=-140]  (0,1.4);
			\node at (0,-1) {$1$}; \draw (0,-.7) -- (0,0);
			\node at (1,-1) {$\dots$};
			\node at (2,-1) {$m$}; \draw (2,-.7) -- (2,0);
			\end{tikzpicture}
		}\end{equation*}\\
		These computations show that $i$ and $r$ are chain homotopy inverses. 
		
		We verify that $\S(n,0)$ has the homology of a point. In fact, the whole complex $\S(n,0)$ is generated by the following element \vspace*{-17pt}
		\begin{center}
			\boxed{\begin{tikzpicture}[scale=.4]
				\node at (0,0){$\bullet$}; \draw (0,0) -- (0,2);
				\node at (1,0){$\bullet$}; \draw (1,0) -- (1,2);
				\node at (3,0){$\bullet$}; \draw (3,0) -- (3,2);
				\node at (2,1){$\cdots$}; 
				\end{tikzpicture}}
		\end{center}\vspace*{-10pt}
		To see this, we consider any basis element in $\S(n,0)$ and follow up any terminal strand. If we encounter \product , then that linear generator is $0$. If we encounter \coproduct\,, then the strand can be removed. We can perform this simplification until the chosen basis element equals the one above.
		
		Finally, if $\psi_{(n,m)}:\S(n,m)\to\overline{\End}(n,m)$ is the biarity $(n,m)$ part of the map, then we can show it is a homology isomorphism with an induction argument using the commutative diagram
		$$
		\xymatrix{
			\S(n,m) \ar[d]_{\psi_{(n,m)}} & \S(n,m+1) \ar[d]^{\psi_{(n,m+1)}} \ar[l]_-{r} \\
			\overline{\End}(k)(n,m) \ar[r]^-{\cong} & \overline{\End}(k)(n,m+1) }
		$$
		where the map in the bottom is induced from the isomorphism $k\cong k\tensor k$.
	\end{proof}
	
	\begin{theorem} \label{sigma-free resolution}
		The prop $\S$ is an $E_\infty$-prop.
		\begin{proof}
			Since by construction the action of $\Sigma_m$ on $U(\S)(m)=\S(1,m)$ is free, the theorem follows from the previous lemma.
		\end{proof}
	\end{theorem}
	
	\begin{remark}
		Notice that the operad obtained by restricting to $\{S(n,1)\}$ is not $\Sigma$-free. For example, 
		\begin{center}
			\boxed{\begin{tikzpicture}[scale=.4]
				\node at (0,0){$\bullet$}; \draw (0,0) -- (0,2); \node at (0,2.5){$1$};
				\node at (1,0){$\bullet$}; \draw (1,0) -- (1,2); \node at (1,2.5){$2$};
				\draw (2,-.05) -- (2,2); \node at (2,2.5){$3$};
				\end{tikzpicture}}
		\end{center}
		in $\S(3,1)$ is fixed by the transposition $(1,2)$.
	\end{remark}	
	
	\section{The simplex category and the prop $\S$}
	In this section we describe a family of $\S$-bialgebra structures on the chains of the standard simplices that are natural with respect to simplicial maps, i.e., natural prop morphisms 
	$$\big\{\S\to\End(\chains(\Delta^d))\big\}_{d\in\Delta\, .}$$
	The image of the three generators of $\S$ are the Alexander-Whitney map, the augmentation map, and a differential graded version of the join of simplices. For any simplicial set $X$ this family induces natural $E_\infty$-structures
	$$U(\S)\to \End(\chains(X))\hspace*{.5cm}\text{and}\hspace*{.5cm}U(\S)\to \End^{\mathrm{op}}(\cochains(X)).$$
	
	\subsection{The natural $\S$-bialgebra on the chains of standard simplices}
	\begin{notation}
		Consider the simplex category $\Delta$ and the functors of normalized chains $\chains(-)$ and cochains $\cochains(-)$ with coefficients in $k$. For every object $d$ in the simplex category we respectively denote the image of $\Hom_{\Delta}(-,d)$ with respect to these functors by $\chains(\Delta^d)$ and $\cochains(\Delta^d)$.
	\end{notation}
	
	\begin{theorem} \label{chain representations}
		For every object $d$ of the simplex category, the following assignment of generators defines a prop morphism $\S\to\End(\chains(\Delta^d))$ which is natural with respect to simplicial maps: \par
		\ \ \ \ $\counit\in\Hom(\chains(\Delta^d),k)_0$ is the augmentation map, i.e. 
		$$\counit\ [v_0,\dots,v_q]=\begin{cases} 1 & \text{ if } q=0 \\ 0 & \text{ if } q>0.\end{cases}$$
		\ \ \ $\coproduct\in\Hom(\chains(\Delta^d),\chains(\Delta^d)^{\otimes2})_0$ is the Alexander-Whitney map, i.e. $$\coproduct\ [v_0,\dots,v_q]=\sum_{i=0}^q[v_0,\dots,v_i]\otimes[v_i,\dots,v_q].$$ 
		\ \ \ $\product\in\Hom(\chains(\Delta^d)^{\otimes2},\chains(\Delta^d))_1$ is the join map, i.e. 
		$$\product\ \big(\left[v_0,\dots,v_p \right] \otimes\left[v_{p+1},\dots,v_q\right]\big)=\begin{cases} (-1)^p\text{sign}(\pi)\cdot\left[v_{\pi(0)},\dots,v_{\pi(q)}\right] & \text{ if } i\neq j \text{ implies } v_i\neq v_j \\ \hspace*{1.15cm}
		0 & \text{ if not} \end{cases}
		$$
		\ \ \ where $\pi$ is the permutation that orders the totally ordered set of vertices.
		\begin{proof}[Proof] Throughout this proof we identify\ \counit\ , \ \coproduct \ ,\ and \product \ with their images in $\End(\chains(\Delta^d))$.
			We need to verify that the relations of $\S$ are satisfied by these images. 
			
			The fact that the four maps\hspace*{.2cm} \leftcounitality \ , \hspace*{.3cm} \rightcounitality \ , \hspace*{.2cm} $\partial$ \coproduct \ , \, and \ $\partial$ \counit \ \ are equal to the corresponding zero map is classical and can be easily verified. 
			
			The join map \product\ is of degree $1$ and the augmentation map \counit\ vanishes on chains of degree greater than $0$, so \productcounit\ is the zero map. 
			
			In order to establish $\partial\ \product=\ \boundary$\;, we must consider six cases for the basis elements which the map is applied to. Recall that the boundary in a $\Hom$ complex is given by $(\partial f)(v) = \partial(f(v)) - (-1)^{\left|f\right|} f(\partial(v))$.
			\begin{enumerate}[leftmargin=.6cm]
				\item Both of degree $0$ and not sharing a vertex\par
				\begin{center}
					\begin{tikzcd}
						& {[v]\tensor[w]} \arrow[swap, dl, maps to, "\product", out=180, in=90] \arrow[d, maps to, "\boundary\ ",swap ] \arrow[dr, maps to, "\displaystyle{\partial}", in=90, out=0]  & \\
						{[v,w]} \arrow[dr, |- , "\displaystyle{\partial}", out=-90, in=180, swap] &
						{[w]-[v]}  & 
						{0}  \arrow[swap, dl, |- , "\product", out=-90, in=0, swap] \\
						& + \uar &
					\end{tikzcd}
				\end{center}
				
				\item Both of degree $0$ and sharing a vertex\par
				\begin{center}
					\begin{tikzcd}
						& {[v]\tensor[v]} \arrow[swap, dl, maps to, "\product", out=180, in=90] \arrow[d, maps to, "\boundary\ ", swap ] \arrow[dr, maps to, "\displaystyle{\partial}", in=90, out=0]  &  \\
						{0} \arrow[dr, |- , "\displaystyle{\partial}", out=-90, in=180, swap] & 
						{0}  & 
						{0}  \arrow[swap, dl, |- , "\product",out=-90, in=0, swap]\\
						& + \uar &
					\end{tikzcd}
				\end{center}
				
				\item Only one of degree $0$ and sharing a vertex\par 		
				\begin{center}
					\begin{tikzcd}[column sep=tiny]
						& {[v]\tensor[v_0,\dots,v_p]} \arrow[swap, dl, maps to, "\product", out=180, in=90] \arrow[d, maps to, "\boundary\ ", swap ] \arrow[dr, maps to, "\displaystyle{\partial}", in=90, out=0]  &	[-10pt]  \\
						{0} \arrow[dr, |- , "\displaystyle{\partial}", out=-90, in=190, swap] &
						{[v_0,\dots,v_p]}  &
						{\sum_{i\geq1}(-1)^{i}[v]\tensor[v_0,\dots,\hat{v}_i,\dots,v_p]}  \arrow[dl, |- , "\product", out=-90, in=-15, swap]  \\
						& + \uar &
					\end{tikzcd}
				\end{center}
				
				\item Only one of degree $0$ and not sharing a vertex\par 	
				\begin{center}
					\begin{tikzcd}[column sep=-5pt]
						& {[v_0]\tensor[v_1,\dots,v_p]} \arrow[swap, dl, maps to, "\product", out=180, in=90] \arrow[swap, d, maps to, "\boundary\ " ] \arrow[dr, maps to, "\displaystyle{\partial}",in=90, out=0]  &	\\
						{sign(\pi)\cdot[v_{\pi(0)},\dots,v_{\pi(p)}]} \arrow[dr, |- , "\displaystyle{\partial}", out=-90, in=190, swap] & 
						{[v_1,\dots,v_p]}  & 
						{\sum_{i\geq1}(-1)^{i-1}[v_0]\tensor[v_1,\dots,\hat{v}_i,\dots,v_p]} \arrow[dl, |- , "\product", out=-90, in=-15, swap] \\
						& + \uar &
					\end{tikzcd}
				\end{center}
				
				\item None of degree $0$ and sharing a vertex\par
				\begin{center}
					\begin{tikzcd}[column sep=tiny]
						& 	{[v_0,\dots,v_p]\tensor[w_0,\dots,w_q]} \arrow[swap, dl, maps to, "\product",out=180, in=90] \arrow[swap, d, maps to, "\boundary\ " ] \arrow[dr, maps to, "\displaystyle{\partial}",in=90, out=0]  & 	[-45pt] \\
						{0} \arrow[dr, |- , "\displaystyle{\partial}", out=-90, in=190, swap]& 
						{0}  &
						\begingroup	\renewcommand*{\arraystretch}{1.5}
						\begin{matrix}\sum_{i}(-1)^{i}[v_0,\dots,\hat v_i,\dots,v_p]\tensor[w_0,\dots,w_q]\ + \\(-1)^p\sum_{j}(-1)^{j}[v_0,\dots,v_p]\tensor[w_0,\dots,\hat{w}_j,\dots,w_q]	\end{matrix}
						\endgroup 	\arrow[dl, |- , "\product", out=-90, in=-15, swap] \\
						& + \uar &
					\end{tikzcd}
				\end{center}
				
				\item None of degree $0$ and not sharing a vertex\par
				\begin{center}
					\begin{tikzcd}[column sep=-45pt]
						& {[v_0,\dots,v_p]\tensor[v_{p+1},\dots,v_q]} \arrow[swap, dl, maps to, "\product", out=180, in=90] \arrow[d, maps to, "\boundary\ ", swap] \arrow[dr, maps to, "\displaystyle{\partial}", in=90, out=0]  & [-0pt]  \\
						{\begin{matrix}	(-1)^p sign(\pi)\ \cdot [v_{\pi(0)},\dots,v_{\pi(q)}]\end{matrix}}  \arrow[dr, |- , "\displaystyle{\partial}", out=-90, in=190, swap] & 
						{0}  & 
						\begingroup
						\renewcommand*{\arraystretch}{1.5}
						\begin{matrix}
							\sum_{i}(-1)^{i}[v_0,\dots,\hat v_i,\dots,v_p]\tensor[v_{p+1},\dots,v_q]\ + \\ 
							\sum_{i}(-1)^{i-1}[v_0,\dots,v_p]\tensor[v_{p+1},\dots,\hat{v}_i,\dots,v_q]
						\end{matrix}
						\endgroup \arrow[dl, |- , "\product", out=-90, in=-15, swap] \\
						& + \uar &
					\end{tikzcd}
				\end{center}
			\end{enumerate}
		\end{proof}
	\end{theorem}

	\subsection{The associated $E_\infty$-coalgebra on the chains of simplicial sets}
	
	\begin{remark}\label{from prop rep to operad rep}
		The natural family of prop morphisms $\{\S\to \End(\chains(\Delta^d))\}_{d\in\Delta}$ of Theorem \ref{chain representations} restricts to a natural family of operad morphisms $\{U(\S)\to \End(\chains(\Delta^d))\}_{d\in\Delta\,.}$
	\end{remark}

	\begin{corollary} \label{simplicial set representations}
		The families of prop morphisms described in Remark \ref{from prop rep to operad rep} induce, for every simplical set $X$, natural operad morphisms 
		$$U(\S)\to \End(\chains(X))\hspace*{1cm}\text{and}\hspace*{1cm}U(\S)\to \End^{\mathrm{op}}(\cochains(X)).$$
	\end{corollary}
	\begin{proof}
		This is a straighforward application of the Kan extension construction. The details can be found for example in \cite{may2003operads}, where such a family is thought of as an operad morphism from $U(\S)$ to the canonical Eilemberg-Zilber operad.  
	\end{proof}
	\begin{remark}
		It is not true that the family of representations described in Theorem \ref{chain representations} induces for every simplical set $X$ a prop morphism $\S\to\End(\chains(X))$. For example, consider the simplicial set containing only two $0$-simplices, whose join cannot be defined. 
	\end{remark}
	
	\begin{example}\label{example (121)[012] a)} For any simplicial set $X$, the following elements in $\S(1,2)$ are sent to representatives of the first three Steenrod cup-i coproducts in $\chains(X)$.
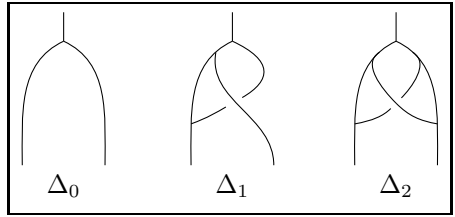
	\begin{center}
		\boxed{\begin{tikzpicture}[scale=.55]
				\draw (1,3.7) to (1,3); 
				
				\draw (1,3) to [out=205, in=90] (0,0);
				\draw (1,3) to [out=-25, in=90] (2,0); 
				
				\node at (1,-.5){$\Delta_0$};
			\end{tikzpicture}\hspace*{1cm}
			\begin{tikzpicture}[scale=.55]
				\draw (1,3.7) to (1,3); 
				
				\draw (1,3) to [out=205, in=90] (0,0);
				
				\draw [shorten >= 0cm] (.6,2.73) to [out=-100, in=90] (2,0);
				
				\draw [shorten >= .15cm] (1,3) to [out=-25, in=30, distance=1.1cm] (1,1.5);
				\draw [shorten <= .1cm] (1,1.5) to [out=210, in=20] (0,1);
				
				\node at (1,-.5){$\Delta_1$};
			\end{tikzpicture}\hspace*{1cm}
			\begin{tikzpicture}[scale=.55]
				\draw (1,3.7) to (1,3); 
				
				\draw (1,3) to [out=205, in=90] (0,0);
				\draw (1,3) to [out=-25, in=90] (2,0); 
				
				\draw [shorten >= 0cm] (.6,2.73) to [out=210, in=135] (1,1.5);
				\draw [shorten <= 0cm] (1,1.5) to [out=-45, in=170] (2,1);
				
				\draw [shorten >= .1cm] (1.4,2.73) to [out=-30, in=45] (1,1.5);
				\draw [shorten <= .1cm] (1,1.5) to [out=-135, in=10] (0,1);
				
				\node at (1,-.5){$\Delta_2$};
			\end{tikzpicture}}
		\end{center}
		
		As an example we compute $\Delta_1[0,1,2]\in\chains(\Delta^2)\tensor\chains(\Delta^2)$. 
		\begin{center}
			\begin{tikzpicture}[scale=.15]
			\node at (-7,1.5){$\Delta_1[0,1,2]=$};
			\draw (1,3.7) to (1,3); 
			
			\draw (1,3) to [out=205, in=90] (0,0);
			
			\draw [shorten >= 0cm] (.6,2.73) to [out=-100, in=90] (2,0);
			
			\draw [shorten >= .15cm] (1,3) to [out=-25, in=30, distance=1.1cm] (1,1.5);
			\draw [shorten <= .1cm] (1,1.5) to [out=210, in=20] (0,1);
			\node at (7.5,1.5){$[0,1,2]=$};
			\end{tikzpicture}
			\begin{tikzpicture}[scale=.15]
			
			\draw (0,3) to [out=-90, in=90] (0,0);
			
			\draw [shorten >= 0cm] (1,3) to [out=-100, in=120] (1,1.5);
			\draw [shorten <= 0cm] (1,1.5) to [out=-60, in=90] (2,0);
			
			\draw [shorten >= .1cm] (2,3) to [out=-90, in=30, distance=1.1cm] (1,1.5);
			\draw [shorten <= .1cm] (1,1.5) to [out=210, in=20] (0,1);	
			
			\node at (3,1.5){$\circ$};
			\draw (5,3.7) to (5,3); 
			\draw [shorten >= 0cm] (4.6,2.73) to [out=-50, in=90] (5,0);
			\draw (5,3) to [out=205, in=90] (4,0);
			\draw (5,3) to [out=-25, in=90] (6,0); 
			\node at (11,1.5){$[0,1,2].$};
			\end{tikzpicture}
		\end{center}	
		\hspace*{-.11cm}\begin{tikzpicture}[scale=.15]
		\node at (-23,1.5){We compute the iterated Alexander-Whitney map };
		\draw (5,3.7) to (5,3); 
		\draw [shorten >= 0cm] (4.6,2.73) to [out=-50, in=90] (5,0);
		\draw (5,3) to [out=205, in=90] (4,0);
		\draw (5,3) to [out=-25, in=90] (6,0); 
		\node at (10.5,1.5){$[0,1,2]$};
		\node at (22,1.5){to be equal to};
		\end{tikzpicture}
		$$\small [0]\tensor[0]\tensor[0,1,2]+[0]\tensor[0,1]\tensor[1,2]+[0]\tensor[0,1,2]\tensor[2]+[0,1]\tensor[1]\tensor[1,2]+[0,1]\tensor[1,2]\tensor[2]+[0,1,2]\tensor[2]\tensor[2].$$
		\hspace*{-.11cm}\begin{tikzpicture}[scale=.15]
		\node at (-6,1.5){Applying};	
		\draw (0,3) to [out=-90, in=90] (0,0);
		
		\draw [shorten >= 0cm] (1,3) to [out=-100, in=120] (1,1.5);
		\draw [shorten <= 0cm] (1,1.5) to [out=-60, in=90] (2,0);
		
		\draw [shorten >= .1cm] (2,3) to [out=-90, in=30, distance=1.1cm] (1,1.5);
		\draw [shorten <= .1cm] (1,1.5) to [out=210, in=20] (0,1);
		
		\node at(8,1.5){to it gives};
		\end{tikzpicture}
		$$-[0,1,2]\tensor[0,1]+[0,2]\tensor[0,1,2]+[0,1,2]\tensor[1,2]$$
		and therefore,
		$$\Delta_1[0,1,2]=-[0,1,2]\tensor[0,1]+[0,2]\tensor[0,1,2]+[0,1,2]\tensor[1,2].$$
		For an axiomatic characterization of these coproducts and their relationship to the nerve of higher dimensional categories see \cite{medina2018axiomatic} and \cite{medina2019globular}. For a state of the art algorithm for their computation and a cochain level proof of the Cartan formula see \cite{medina2018persistence} and \cite{medina2019effective}.
	\end{example}
	
	\appendix
	
	\section{The prop $\MS$ and the Surjection operad}
	
	In this appendix we construct $\MS$ a quotient prop of $\S$, also finitely presented, whose associated operad is isomorphic up to signs to the Surjection operad $Sur$ of McClure-Smith \cite{mcclure2003multivariable} and Berger-Fresse \cite{berger2004combinatorial}.
	
	Additionally, we identify explicitly a suboperad $U(\Sh)$ of $U(\S)$ making the following diagram commute
	\begin{center}
		\begin{tikzcd}
		U(\S)  \arrow[r, "U(\varphi)"] & \End(\chains(X); \F_2) \\
		U(\Sh) \arrow[u, hook] \arrow[d, two heads] &        \\
		U(\MS) \arrow[r, "\cong"] & Sur \arrow[uu, "\phi"']
		\end{tikzcd}
	\end{center}
	for every simplicial set $X$, where $\phi$ is the $Sur$-coalgebra structure defined by these authors and $U(\phi)$ is the one defined in Corrolary \ref{simplicial set representations}.
	
	\subsection{The prop $\MS$}
	
	\begin{definition}\label{Prop MS}
		Let $\MS$ be defined by the same presentation as $\S$ together with the commutative, associative, coassociative, Leibniz, and involutive relations
		\begin{center}	
			\boxed
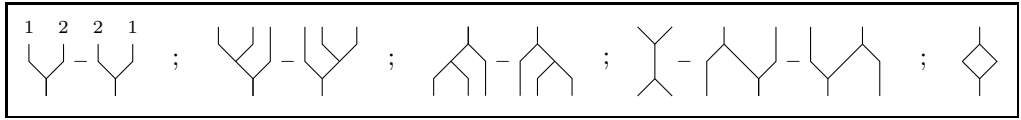
{
				\begin{tikzpicture}[scale=.23]
				\node at (4.8,4.2){};
				\node at (4.8,-.2){};
				
				\node at (5,4){$\scriptstyle 1$};
				\node at (7,4){$\scriptstyle 2$};
				\draw (5,3)--(5,2)--(6,1)--(6,0);
				\draw (7,3)--(7,2)--(6,1)--(6,0);
				
				\node at (8,2){$\scriptstyle -$};
				
				\node at (9,4){$\scriptstyle 2$};
				\node at (11,4){$\scriptstyle 1$};
				\draw (9,3)--(9,2)--(10,1)--(10,0);
				\draw (11,3)--(11,2)--(10,1)--(10,0);
				
				\node at (13.5,2){;};
				\end{tikzpicture}\ \ 
				\begin{tikzpicture}[scale=.23]
				\node at (4.8,4.2){};
				\node at (4.8,-.2){};
				
				\draw (7,4)--(7,3)--(6,2);
				\draw (5,4)--(5,3)--(7,1)--(7,0);
				\draw (8,4)--(8,2)--(7,1)--(7,0);
				\node at (9,2){$\scriptstyle -$};
				
				\draw (10,4)--(10,2)--(11,1)--(11,0);
				\draw (13,4)--(13,3)--(11,1);
				\draw (11,4)--(11,3)--(12,2);
				
				\node at (15,2){;};
				\end{tikzpicture}\ \ 
				\begin{tikzpicture}[scale=.23]
				\node at (4.8,-4.2){};
				\node at (4.8,.2){};
				
				\draw (7,-4)--(7,-3)--(6,-2);
				\draw (5,-4)--(5,-3)--(7,-1)--(7,0);
				\draw (8,-4)--(8,-2)--(7,-1)--(7,0);
				\node at (9,-2){$\scriptstyle -$};
				
				\draw (10,-4)--(10,-2)--(11,-1)--(11,0);
				\draw (13,-4)--(13,-3)--(11,-1);
				\draw (11,-4)--(11,-3)--(12,-2);
				
				\node at (15,-2){;};
				\end{tikzpicture}\ \ 
				\begin{tikzpicture}[scale=.23]
				\draw (0,4)--(1,3)--(1,1)--(0,0);
				\draw (2,4)--(1,3);
				\draw (1,1)--(2,0);
				\node at (2.7,2){$\scriptstyle -$};
				
				\draw (5,4)--(5,3)--(4,2)--(4,0);
				\draw (5,3)--(7,1)--(7,0);
				\draw (8,4)--(8,2)--(7,1)--(7,0);
				\node at (9,2){$\scriptstyle -$};
				
				\draw (10,4)--(10,2)--(11,1)--(11,0);
				\draw (13,4)--(13,3)--(11,1);
				\draw (13,3)--(14,2)--(14,0);
				
				\node at (4.8,4.2){};
				\node at (15.3,-.2){};
				\end{tikzpicture}
				\begin{tikzpicture}[scale=.23]
				\node at (-2,2){; \ };
				\draw (1,4)--(1,3)--(0,2)--(1,1)--(1,0);
				\draw (1,3)--(2,2)--(1,1);
				
				\node at (2.2,4.2){};
				\node at(2.2,-.2){};
				\end{tikzpicture}}
		\end{center}	
		We regard $\MS$ as a quotient of $\S$ ans consider the induced quotient operad morphism $U(\S) \twoheadrightarrow U(\MS)$.
	\end{definition}
	
	\begin{notation} \label{notation: higher valence}
		We will utilize the following diagramatic simplification
		\begin{center}
			\boxed{\begin{tikzpicture}[scale=.4]
				\draw (6,2)--(7,1)--(7,0);
				\draw (8,2)--(7,1);
				\node at (6,2.5){$\scriptstyle 1$};
				\node at (7,2.5){$\scriptstyle \dots$};
				\node at (8,2.5){$\scriptstyle n$};
				
				\draw (11,.5)--(12,1.5)--(12,2.5);
				\draw (13,.5)--(12,1.5);
				\node at (11,0){$\scriptstyle 1$};
				\node at (12,0){$\scriptstyle \dots$};
				\node at (13,0){$\scriptstyle n$};
				\end{tikzpicture}
				\begin{tikzpicture}[scale=.4]
				\node at (-1,2.5){};
				\node at (1,1.5){or};
				\node at (1,0){};
				\end{tikzpicture}\qquad 
				\begin{tikzpicture}[scale=.4]		
				\draw (12.5,2.5)--(13,2);
				\draw (14.5,2.5)--(14,2);
				\node at (13.5,1.5){$\scriptstyle n-1$};
				\draw (13.5,1)--(13.5,0);
				
				\draw (17.5,0)--(18,.5);
				\draw (19.5,0)--(19,.5);
				\node at (18.5,1){$\scriptstyle n-1$};
				\draw (18.5,1.5)--(18.5,2.5);
				
				\node at (19.6,-.15){};
				\node at (12.4,2.65){};
				\end{tikzpicture}}
		\end{center}
		to represent labeled directed graphs resulting from iterated grafting of the product and coproduct in the left comb order
		\begin{center}
			\boxed
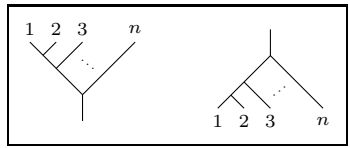
{\begin{tikzpicture}[scale=.35]		
				\node at (-2.15,-.25){};
				\node at (-2.15,3.25){};
				
				\node at (-2,3.5){$\scriptstyle 1$};
				\node at (-1,3.5){$\scriptstyle 2$};
				\node at (0,3.5){$\scriptstyle 3$};
				\node at (2,3.5){$\scriptstyle n$};
				
				\draw (-2,3)--(0,1);
				\draw (2,3)--(0,1)--(0,0);
				\draw (0,3)--(-1,2);
				\draw (-1,3)--(-1.5,2.5);
				\draw (.2,2.3) node[scale= 0.5] {$\ddots$};
				\end{tikzpicture}
				\qquad 
				\begin{tikzpicture}[scale=.35]	
				
				\node at (-2,-3.5){$\scriptstyle 1$};
				\node at (-1,-3.5){$\scriptstyle 2$};
				\node at (0,-3.5){$\scriptstyle 3$};
				\node at (2,-3.5){$\scriptstyle n$};
				
				\draw (-2,-3)--(0,-1);
				\draw (2,-3)--(0,-1)--(0,0);
				\draw (0,-3)--(-1,-2);
				\draw (-1,-3)--(-1.5,-2.5);
				\draw (.2,-2.3) node[scale= 0.5, rotate = 75] {$\ddots$};
				\end{tikzpicture}}
		\end{center}
		with the case $n=1$ representing the identity element.
	\end{notation}	
	
	\begin{definition} \label{definition: surjection-like element}
		An element in $\S(1,m)$ is called \textbf{surjection-like} if it is of the form
		\begin{center}
			\boxed
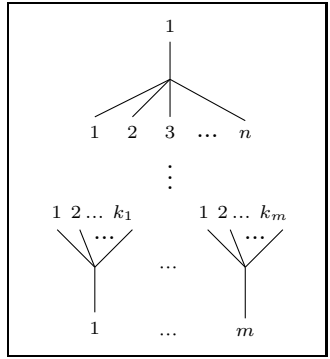
{	
				\begin{tikzpicture}[scale=1]
				\draw (0,0)--(0,-.6) node[below]{$\scriptstyle 1$};
				\draw (0,0)--(.5,.5);
				\draw (0,0)-- (-.2,.5) node[above]{\quad $\scriptstyle 1\ 2\, ...\ k_1$};
				\draw (-.5,.5)--(0,0);
				\node at (.13,.4){$...$};
				
				\node at (1,0){$\scriptstyle \cdots$};
				\node at (1,-.9){$\scriptstyle \cdots$};
				
				\draw (2,0)--(2,-.68) node[below]{$\scriptstyle m$};
				\draw (2,0)--(2.5,.5);
				\draw (2,0)-- (1.8,.5) node[above]{\quad $\scriptstyle 1\ 2\, ...\ k_m$};
				\draw (1.5,.5)--(2,0);
				\node at (2.13,.4){$...$};
				
				\draw (1,2.5)--(1,3) node[above]{$\scriptstyle 1$};
				\draw (1,2.5)--(0,2) node[below]{$\scriptstyle 1$};
				\draw (1,2.5)--(.5,2) node[below]{$\scriptstyle 2$};
				\draw (1,2.5)--(1,2) node[below]{$\scriptstyle 3$};
				\draw (1,2.5)--(2,1.95) node[below]{$\scriptstyle n$};
				\node at (1.5,1.75){$...$};
				
				\node at (1,1.3) {$\vdots$};
				
				\node at (2.85,0){};
				\end{tikzpicture}
			}
		\end{center} 
		with no internal vertices, in particular $\sum_{i=1}^m k_i=n$, and, for every $1\leq r\leq m$, the associated map $\overline{k}_r\to\overline{n}$ is order preserving. The surjection $\overline{n} \to \overline{m}$ defined by this surjection-like element is referred to as its \textbf{associated surjection}. 
	\end{definition}
	
	Let us consider the following pairs which we refer to as \textbf{rewriting rules}:
	\begin{equation*}
	\boxed{	\begin{aligned}
		\begin{tikzpicture}[scale=.2]
		\node at (2.5,-.2){};
		
		\draw (3,4)--(3,3)--(4,2)--(4,0);
		\draw (5,4)--(5,3)--(4,2);
		
		\draw [fill] (4,0.3) circle [radius=0.3];
		
		\draw[myptr] (6,2)--(8,2);
		\node at (9.5,2){\large 0};
		
		\node at (13,2){;};
		\end{tikzpicture}
		\begin{tikzpicture}[scale=.2]
		\node at (3.3, -4.2){};
		\node at (1.3,.2){};
		
		\draw (3,-4)--(3,-3)--(4,-2)--(4,0);
		\draw (5,-4)--(5,-3)--(4,-2);
		
		\draw [fill] (3,-3.7) circle [radius=0.3];
		
		\draw[myptr] (6,-2)--(8,-2);
		
		\draw (9.5, -4)--(9.5,0);
		
		\node at (13,-2){;};
		\end{tikzpicture}
		\begin{tikzpicture}[scale=.2]
		\node at (3.3, -4.2){};
		\node at (1.3,.2){};
		
		\draw (3,-4)--(3,-3)--(4,-2)--(4,0);
		\draw (5,-4)--(5,-3)--(4,-2);
		
		\draw [fill] (5,-3.7) circle [radius=0.3];
		
		\draw[myptr] (6,-2)--(8,-2);
		
		\draw (9.5, -4)--(9.5,0);
		
		\node at (13,-2){;};
		\end{tikzpicture}
		\begin{tikzpicture}[scale=.2]
		\node at (-1.2,0){};
		
		\draw (1,4)--(1,3)--(0,2)--(1,1)--(1,0);
		\draw (1,3)--(2,2)--(1,1);
		
		\draw[myptr] (3,1.9)--(4.8,1.9);
		\node at(6.2,1.9){\Large 0};
		
		\node at (2.2,4.2){};
		\node at(2.2,-.2){};
		
		\node at (9,2){;};
		\end{tikzpicture}
		\begin{tikzpicture}[scale=.2]
		\node at (-1.2,0){};
		
		\draw (1,4)--(1,3)--(0.5,2.5)--(1.5,1.5)--(1,1)--(1,0);
		\draw (1,3)--(1.5,2.5)--(1.2,2.2);
		\draw (1,1)--(0.5,1.5)--(.8,1.8);
		
		\draw[myptr] (3,1.9)--(4.8,1.9);
		\node at(6.2,1.9){\Large 0};
		
		\node at (2.2,4.2){};
		\node at(2.2,-.2){};
		\end{tikzpicture} \\
		\begin{tikzpicture}[scale=.2]
		\node at (4.1, 4.2){};
		\node at (2.1,-.2){};
		
		\node at (3,4){$\scriptstyle 2$};
		\node at (5,4){$\scriptstyle 1$};
		\draw (3,3)--(3,2)--(4,1)--(4,0);
		\draw (5,3)--(5,2)--(4,1)--(4,0);
		
		\draw[myptr] (6,2)--(8,2);
		
		\node at (9,4){$\scriptstyle 1$};
		\node at (11,4){$\scriptstyle 2$};
		\draw (9,3)--(9,2)--(10,1)--(10,0);
		\draw (11,3)--(11,2)--(10,1)--(10,0);
		
		\node at (13.5,2){;};
		\end{tikzpicture} \ \ 
		\begin{tikzpicture}[scale=.2]
		\node at (4.8,4.2){};
		\node at (4.8,-.2){};
		
		\draw (4,4)--(4,2)--(5,1)--(5,0);
		\draw (7,4)--(7,3)--(5,1);
		\draw (5,4)--(5,3)--(6,2);
		
		\draw[myptr] (8,2)--(10,2);
		
		\draw (13,4)--(13,3)--(12,2);
		\draw (11,4)--(11,3)--(13,1)--(13,0);
		\draw (14,4)--(14,2)--(13,1)--(13,0);
		
		\node at (16,2){;};
		\end{tikzpicture}\ \ 
		\begin{tikzpicture}[scale=.2]
		\node at (4.8,-4.2){};
		\node at (4.8,.2){};
		
		\draw (4,-4)--(4,-2)--(5,-1)--(5,0);
		\draw (7,-4)--(7,-3)--(5,-1);
		\draw (5,-4)--(5,-3)--(6,-2);
		
		\draw[myptr] (8,-2)--(10,-2);
		
		\draw (13,-4)--(13,-3)--(12,-2);
		\draw (11,-4)--(11,-3)--(13,-1)--(13,0);
		\draw (14,-4)--(14,-2)--(13,-1)--(13,0);
		
		\node at (16,-2){;};
		\end{tikzpicture}\ \ 
		\begin{tikzpicture}[scale=.2]
		\draw (-2,4)--(-1,3)--(-1,1)--(-2,0);
		\draw (0,4)--(-1,3);
		\draw (-1,1)--(0,0);
		
		\draw[myptr] (.7,2)--(2.7,2);
		
		\draw (5,4)--(5,3)--(4,2)--(4,0);
		\draw (5,3)--(7,1)--(7,0);
		\draw (8,4)--(8,2)--(7,1)--(7,0);
		
		\node at (9,2){$\scriptstyle +$};
		
		\draw (10,4)--(10,2)--(11,1)--(11,0);
		\draw (13,4)--(13,3)--(11,1);
		\draw (13,3)--(14,2)--(14,0);
		
		\node at (4.8,4.2){};
		\node at (15.3,-.2){};
		\end{tikzpicture}
		\end{aligned}}
	\end{equation*}
	Following the terminology of Gr\"obner basis \cite{dotsenko2010grobner,loday2012algebraic}, a \textbf{critical monomial} associated to these rewriting rules is defined to be the composition of two elements in the domain of the rules. A critical monomial is said to be \textbf{confluent} if the order in which the rewriting is performed does not affect the outcome. 
	
	\begin{lemma} \label{All critical monomials are confluent}
		The rewriting rules above yield only confluent critical monomials.
	\end{lemma}
	
	\begin{proof}
		We start by checking the confluence of the critical monomials associated to the composing of the Leibniz relation and the involution relation
		\begin{equation*}
		\boxed
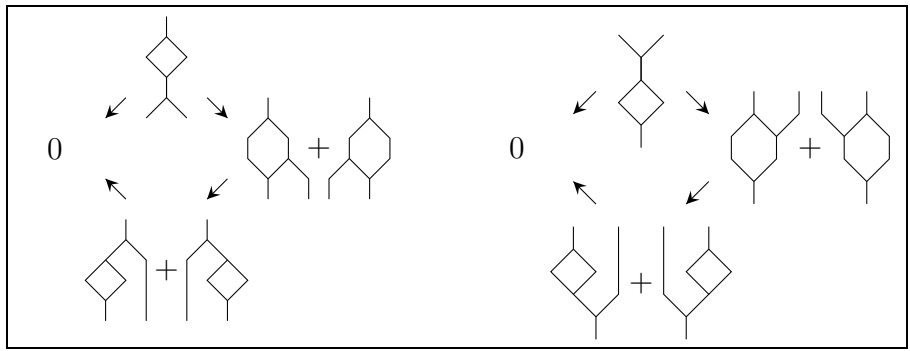
{
			\begin{tikzpicture}[scale=.27]
			\node at (0,12) {};
			\draw (3,10)--(4,11)--(4,12)--(3,13)--(4,14)--(4,15);
			\draw (5,10)--(4,11)--(4,12)--(5,13)--(4,14);
			
			\draw (9,6)--(9,7)--(8,8)--(8,9)--(9,10)--(9,11);
			\draw (9,7)--(10,8)--(10,9)--(9,10);
			\draw (11,6)--(11,7)--(10,8);
			
			\draw (12,6)--(12,7)--(13,8)--(13,9)--(14,10)--(14,11);
			\draw (14,6)--(14,7)--(13,8);
			\draw (14,7)--(15,8)--(15,9)--(14,10);
			
			\draw (1,0)--(1,1)--(0,2)--(1,3)--(2,4)--(2,5);
			\draw (1,1)--(2,2)--(1,3);
			\draw (3,0)--(3,3)--(2,4);
			
			\draw (5,0)--(5,3)--(6,4)--(6,5);
			\draw (7,0)--(7,1)--(6,2)--(7,3)--(6,4);
			\draw (7,1)--(8,2)--(7,3);
			
			\draw[myptr] (2,6)--(1,7);	
			\draw[myptr] (7,7)--(6,6);	
			\draw[myptr] (6,11)--(7,10);	
			\draw[myptr] (2,11)--(1,10);
			
			\node at (-1.5,8.5){\Large 0};		
			\node at (11.5,8.5){\Large +};
			\node at (4,2.5){\Large +};
			
			\node at (-3,3){};
			\node at (15,-.5){};
			\end{tikzpicture}
			\qquad
			\begin{tikzpicture}[scale=.3]
			\node at (0,14) {};
			\draw (3,13.5)--(4,12.5)--(4,11.5)--(3,10.5)--(4,9.5)--(4,8.5);
			\draw (5,13.5)--(4,12.5)--(4,11.5)--(5,10.5)--(4,9.5);
			
			\draw (9,11)--(9,10)--(8,9)--(8,8)--(9,7)--(9,6);
			\draw (9,10)--(10,9)--(10,8)--(9,7);
			\draw (11,11)--(11,10)--(10,9);
			
			\draw (12,11)--(12,10)--(13,9)--(13,8)--(14,7)--(14,6);
			\draw (14,11)--(14,10)--(13,9);
			\draw (14,10)--(15,9)--(15,8)--(14,7);
			
			\draw (1,5)--(1,4)--(0,3)--(1,2)--(2,1)--(2,0);
			\draw (1,4)--(2,3)--(1,2);
			\draw (3,5)--(3,2)--(2,1);
			
			\draw (5,5)--(5,2)--(6,1)--(6,0);
			\draw (7,5)--(7,4)--(6,3)--(7,2)--(6,1);
			\draw (7,4)--(8,3)--(7,2);
			
			\draw[myptr] (2,6)--(1,7);	
			\draw[myptr] (7,7)--(6,6);	
			\draw[myptr] (6,11)--(7,10);	
			\draw[myptr] (2,11)--(1,10);
			
			\node at (-1.5,8.5){\Large 0};		
			\node at (11.5,8.5){\Large +};
			\node at (4,2.5){\Large +};
			
			\node at (-4,3){};
			\node at (15,3){};
			\end{tikzpicture} 
		}
		\end{equation*}
		We do the same for the counitality and Leibniz relations
		\begin{equation*}
		\boxed
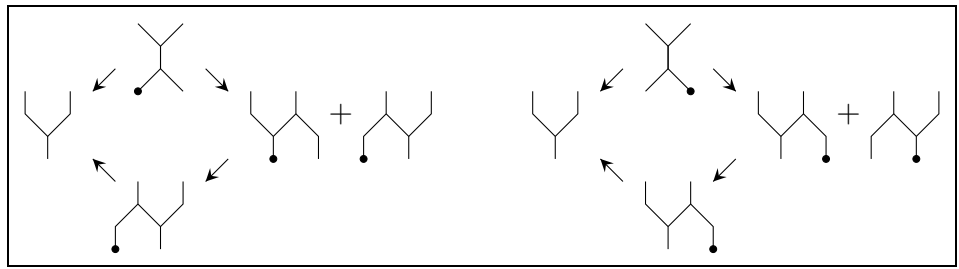
{
			\begin{tikzpicture}[scale=.3]
			\draw (0,0)--(0,1)--(1,2)--(1,3);
			\draw (2,0)--(2,1)--(1,2);
			\draw (2,1)--(3,2)--(3,3);
			
			\draw (7,4)--(7,5)--(6,6)--(6,7);
			\draw (7,5)--(8,6)--(8,7);
			\draw (9,4)--(9,5)--(8,6);
			
			\draw (11,4)--(11,5)--(12,6)--(12,7);
			\draw (13,4)--(13,5)--(12,6);
			\draw (13,5)--(14,6)--(14,7);
			
			\draw (1,7)--(2,8)--(2,9)--(1,10);
			\draw (3,7)--(2,8)--(2,9)--(3,10);
			
			\draw (-3,4)--(-3,5)--(-4,6)--(-4,7);
			\draw (-3,5)--(-2,6)--(-2,7);
			
			\draw[myptr] (0,3)--(-1,4);	
			\draw[myptr] (5,4)--(4,3);	
			\draw[myptr] (4,8)--(5,7);	
			\draw[myptr] (0,8)--(-1,7);
			
			\node at (10,6){\Large +};
			
			\draw[fill] (1,7) circle (1ex);
			\draw[fill] (7,4) circle (1ex);
			\draw[fill] (0,0) circle (1ex);
			\draw[fill] (11,4) circle (1ex);
			
			\node at (-4,10){};
			\node at (14,0){};
			\end{tikzpicture}
			\qquad\quad\,
			\begin{tikzpicture}[scale=.3]
			\draw (2,0)--(2,1)--(1,2)--(1,3);
			\draw (2,1)--(3,2)--(3,3);
			\draw (4,0)--(4,1)--(3,2);
			
			\draw (7,4)--(7,5)--(6,6)--(6,7);
			\draw (7,5)--(8,6)--(8,7);
			\draw (9,4)--(9,5)--(8,6);
			
			\draw (11,4)--(11,5)--(12,6)--(12,7);
			\draw (13,4)--(13,5)--(12,6);
			\draw (13,5)--(14,6)--(14,7);
			
			\draw (1,7)--(2,8)--(2,9)--(1,10);
			\draw (3,7)--(2,8)--(2,9)--(3,10);
			
			\draw (-3,4)--(-3,5)--(-4,6)--(-4,7);
			\draw (-3,5)--(-2,6)--(-2,7);
			
			\draw[myptr] (0,3)--(-1,4);	
			\draw[myptr] (5,4)--(4,3);	
			\draw[myptr] (4,8)--(5,7);	
			\draw[myptr] (0,8)--(-1,7);
			
			\node at (10,6){\Large +};
			
			\draw[fill] (3,7) circle (1ex);
			\draw[fill] (9,4) circle (1ex);
			\draw[fill] (4,0) circle (1ex);
			\draw[fill] (13,4) circle (1ex);
			
			\node at (-4,10){};
			\node at (14,0){};
			\end{tikzpicture}		
		}
		\end{equation*}
		the coassociative and Leibniz relations
		\begin{equation*}
		\boxed
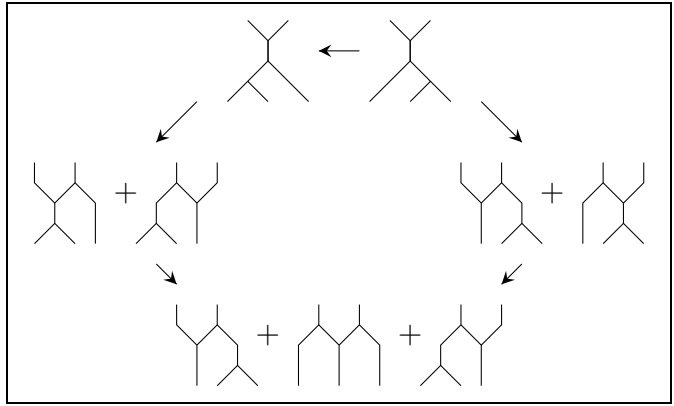
{
			\begin{tikzpicture}[scale=.27]
			\draw (-1.5,14)--(-3.5,16)--(-3.5,17)--(-2.5,18);
			\draw (-3.5,14)--(-4.5,15)--(-3.5,16)--(-3.5,17)--(-4.5,18);						
			\draw (-5.5,14)--(-4.5,15);
			
			\draw (1.5,14)--(3.5,16)--(3.5,17)--(2.5,18);
			\draw (3.5,14)--(4.5,15)--(3.5,16)--(3.5,17)--(4.5,18);						
			\draw (5.5,14)--(4.5,15);
			
			\draw (-7,7)--(-7,9)--(-6,10)--(-6,11);
			\draw (-7,9)--(-8,10)--(-8,11);
			\draw (-8,7)--(-9,8)--(-9,9)--(-8,10);
			\draw (-10,7)--(-9,8);
			
			\draw (-12,7)--(-12,9)--(-13,10)--(-13,11);
			\draw (-13,7)--(-14,8)--(-14,9)--(-13,10);
			\draw (-15,7)--(-14,8);
			\draw (-14,9)--(-15,10)--(-15,11);
			
			\draw (-4,0)--(-5,1)--(-5,2)--(-6,3)--(-6,4);
			\draw (-6,0)--(-5,1);
			\draw (-7,0)--(-7,2)--(-6,3);
			\draw (-7,2)--(-8,3)--(-8,4);
			
			\draw (-2,0)--(-2,2)--(-1,3)--(-1,4);
			\draw (0,0)--(0,2)--(-1,3);
			\draw (0,2)--(1,3)--(1,4);			
			\draw (2,0)--(2,2)--(1,3);
			
			\draw (4,0)--(5,1)--(5,2)--(6,3)--(6,4);
			\draw (6,0)--(5,1);
			\draw (7,0)--(7,2)--(6,3);
			\draw (7,2)--(8,3)--(8,4);
			
			\draw (7,7)--(7,9)--(6,10)--(6,11);
			\draw (7,9)--(8,10)--(8,11);
			\draw (8,7)--(9,8)--(9,9)--(8,10);
			\draw (10,7)--(9,8);
			
			\draw (12,7)--(12,9)--(13,10)--(13,11);
			\draw (13,7)--(14,8)--(14,9)--(13,10);
			\draw (15,7)--(14,8);
			\draw (14,9)--(15,10)--(15,11);
			
			\draw[myptr] (7,14)--(9,12);	
			\draw[myptr] (9,6)--(8,5);	
			\draw[myptr] (-7,14)--(-9,12);	
			\draw[myptr] (-9,6)--(-8,5);
			\draw[myptr] (1,16.5)--(-1,16.5);
			
			\node at (10.5,9.5){\Large +};
			\node at (-10.5,9.5){\Large +};
			\node at (3.5,2.5){\Large +};
			\node at (-3.5,2.5){\Large +};
			
			\node at (15.5,18){};
			\node at (-15.5,0){};
			\end{tikzpicture}
		}
		\end{equation*}
		and the associative and Leibniz relations
		\begin{equation*}
		\boxed
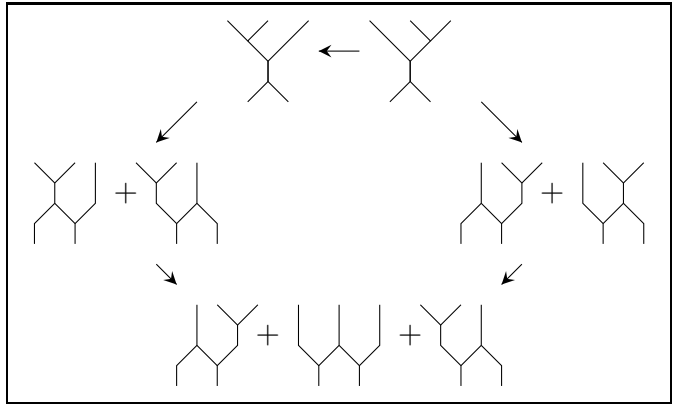
{
			\begin{tikzpicture}[scale=.27]
			\draw (-1.5,18)--(-3.5,16)--(-3.5,15)--(-2.5,14);
			\draw (-3.5,18)--(-4.5,17)--(-3.5,16)--(-3.5,15)--(-4.5,14);						
			\draw (-5.5,18)--(-4.5,17);
			
			\draw (1.5,18)--(3.5,16)--(3.5,15)--(2.5,14);
			\draw (3.5,18)--(4.5,17)--(3.5,16)--(3.5,15)--(4.5,14);						
			\draw (5.5,18)--(4.5,17);
			
			\draw (-7,11)--(-7,9)--(-6,8)--(-6,7);
			\draw (-7,9)--(-8,8)--(-8,7);
			\draw (-8,11)--(-9,10)--(-9,9)--(-8,8);
			\draw (-10,11)--(-9,10);
			
			\draw (-12,11)--(-12,9)--(-13,8)--(-13,7);
			\draw (-13,11)--(-14,10)--(-14,9)--(-13,8);
			\draw (-15,11)--(-14,10);
			\draw (-14,9)--(-15,8)--(-15,7);
			
			\draw (-4,4)--(-5,3)--(-5,2)--(-6,1)--(-6,0);
			\draw (-6,4)--(-5,3);
			\draw (-7,4)--(-7,2)--(-6,1);
			\draw (-7,2)--(-8,1)--(-8,0);
			
			\draw (-2,4)--(-2,2)--(-1,1)--(-1,0);
			\draw (0,4)--(0,2)--(-1,1);
			\draw (0,2)--(1,1)--(1,0);			
			\draw (2,4)--(2,2)--(1,1);
			
			\draw (4,4)--(5,3)--(5,2)--(6,1)--(6,0);
			\draw (6,4)--(5,3);
			\draw (7,4)--(7,2)--(6,1);
			\draw (7,2)--(8,1)--(8,0);
			
			\draw (7,11)--(7,9)--(6,8)--(6,7);
			\draw (7,9)--(8,8)--(8,7);
			\draw (8,11)--(9,10)--(9,9)--(8,8);
			\draw (10,11)--(9,10);
			
			\draw (12,11)--(12,9)--(13,8)--(13,7);
			\draw (13,11)--(14,10)--(14,9)--(13,8);
			\draw (15,11)--(14,10);
			\draw (14,9)--(15,8)--(15,7);
			
			\draw[myptr] (7,14)--(9,12);	
			\draw[myptr] (9,6)--(8,5);	
			\draw[myptr] (-7,14)--(-9,12);	
			\draw[myptr] (-9,6)--(-8,5);
			\draw[myptr] (1,16.5)--(-1,16.5);
			
			\node at (10.5,9.5){\Large +};
			\node at (-10.5,9.5){\Large +};
			\node at (3.5,2.5){\Large +};
			\node at (-3.5,2.5){\Large +};
			
			\node at (15.5,18){};
			\node at (-15.5,0){};
			\end{tikzpicture}
		}
		\end{equation*}
		The remaining pairs of relations are straightforward and left to the interested reader. 
	\end{proof}
	
	\begin{lemma} \label{Basis of surjection-like elements}
		A linear basis of $U(\MS)$ is represented by all surjection-like elements containing no copy of the involution graph
		\begin{equation*}
		\boxed
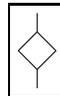
{\begin{tikzpicture}[scale=.25]
		\draw (1,4)--(1,3)--(0,2)--(1,1)--(1,0);
		\draw (1,3)--(2,2)--(1,1);
		\end{tikzpicture}}
		\end{equation*}
	\end{lemma}
	
	\begin{proof}
		We begin by noticing that, by definition, a surjection-like element contains no copies of 
	\begin{equation*}
	\boxed{\ 
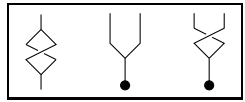
\begin{tikzpicture}[scale=.2]
		\draw (0,5)--(0,4)--(-1,3);
		\draw (-1,3) to (-.2,2.6);
		\draw (1,2)--(.2,2.4);
		\draw (1,2)--(0,1)--(0,0);
		\draw (0,4)--(1,3)--(-1,2)--(0,1);
		\end{tikzpicture}	
		\qquad
		\begin{tikzpicture}[scale=.2]
		\draw (-1,5)--(-1,3)--(0,2)--(0,0);
		\draw (1,5)--(1,3)--(0,2);
		\draw [fill] (0,.2) circle [radius=0.3];
		\end{tikzpicture}
		\qquad
		\begin{tikzpicture}[scale=.2]
		\draw (-1,5)--(-1,4);
		\draw (0,2)--(0,0);
		\draw (1,5)--(1,4)--(-1,3)--(0,2);
		\draw [fill] (0,.2) circle [radius=0.3];
		
		\draw (-1,4) to (-.2,3.6);
		\draw (0,2)--(1,3)--(.2,3.4);
		\end{tikzpicture}\ }
	\end{equation*}
		Therefore, since there are no other relations due to Lemma \ref{All critical monomials are confluent}, a non-zero linear combination of surjection-like elements is mapped to $0$ via $U(\S) \to U(\MS)$, if and only if each surjection-like element contains a copy of the involution graph. 
		
		The fact that each critical monomial of our rewriting rules is confluent, implies that applying them to an arbitrary representative of an element in $U(\MS)$ will produce a unique preferred representative. It is straightforward to see that such representative is a linear combination of surjection-like elements.
	\end{proof}

	\subsection{Comparing the Surjection operad and $U(\MS)$} 
	
	The following definition is due to McClure-Smith and Berger-Fresse. We refer to the resulting operad as the \textbf{Surjection operad}. We present it over the field with two elements $\F_2$ and remit the interested reader to \cite{mcclure2003multivariable} and \cite{berger2004combinatorial} for their respective sign conventions. 
	
	\begin{definition}\label{Sur}
		For a fixed $m\geq1$, consider the free differential graded $\F_2$-module with a basis given by all functions $s:\overline{n}\to\overline{m}$ with any $n\geq1$. The degree of a basis element $s:\overline{n}\to\overline{m}$ is $(n-m)$ and its differential is defined to be $$\partial s=\sum_{k=1}^{n}s\circ\iota_k$$ where $\iota_k:\overline{n-1} \to \overline{n}$ is the order preserving injection that misses $k$. Consider the free differential graded $\F_2$-submodule generated by functions $s:\overline{n}\to\overline{m}$ which are either non-surjective or for which $s(i)$ equals $s(i+1)$ for some $i$. Define $Sur(m)$ to be the associated quotient. The collection $Sur=\{Sur(m)\}_{m\geq1}$ is a $\Sigma$-module with the action of $\Sigma_m$ on $Sur(m)$ given by postcomposition. The $\Sigma$-module $Sur$ is an operad with partial composition $\circ_r: Sur(m')\tensor Sur(m)\to Sur(m+m'-1)$ defined on two generators $s:\overline{n}\to\overline{m}$ and $s':\overline{n'}\to\overline{m'}$ as follows. Represent the surjections $s$ and $s'$ by sequences $(s(1),\dots,s(n))$ and $(s'(1),\dots,s'(n'))$ and suppose that $r$ appears $k$ times in the sequence representing $s'$ as $s'(i_1),\dots,s'(i_k)$. Denote the set of all tuples
		\begin{equation*}
		1 = j_0 \leq j_1 \leq \dots \leq j_k = n
		\end{equation*}
		by $J(k,n)$ and for each such tuple consider the subsequences 
		\begin{equation*}
		(s(j_0),\dots,s(j_1))\ \ \ (s(j_1),\dots,s(j_2))\ \ \cdots\ \ (s(j_{k-1}),\dots,s(j_k)).
		\end{equation*}
		Then, in $(s'(1),\dots,s'(n))$, replace the term $s'(i_t)$ by the sequence $(s(j_{t-1}),\dots,s(j_t))$. In addition, increase the terms $s(j)$ by $r-1$ and the terms $s'(i)$ such that $s'(i)>r$ by $m-1$. The surjection $s'\circ_r s$ is represented by the sum, parametrized by $J(k,n)$, of these resulting sequences.
	\end{definition}
	
	\begin{theorem}\label{Isomorphism MS and Sur}
		The assignment mapping a surjection-like element to its associated surjection induces an isomorphism from $U(\MS)$ to $Sur$.
	\end{theorem}
	
	We start by noticing that a surjection-like element is sent to a surjection with a pair of consecutive integers having the same value if and only if a copy of the involutive graph is contained in it. As described in \mbox{Lemma \ref{Basis of surjection-like elements}}, the surjection-like elements containing no copies of the involution graph represent a basis of $\MS$. The assignment then induces a bijection between linear bases of $U(\MS)$ and $Sur$. This bijection is easily seen to respect degree, differential and symmetric action. The fact it respect operadic composition will follow from the next result.
	\begin{lemma}\label{splitting lemma}
		For any pair $k,n\geq 1$ denote by $A(k,n)$ the set of all sequences $a=\langle a_1,\dots,a_k\rangle$ of non-negative integers satisfying $1+a_1+\dots+a_k=n$ and for any such sequence $a=\langle a_1,\dots,a_k\rangle$ consider the element in $(\MS)(k,n)$ defined by
		\begin{center}
			\boxed
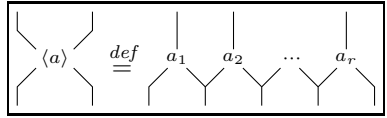
{\begin{tikzpicture}[scale=.25]
				\draw (0,5)--(0,4)--(1,3);
				\draw (4,5)--(4,4)--(3,3);
				\node at (2,2.5){$\scriptstyle \langle a\rangle$};
				\draw (1,2)--(0,1)--(0,0); 
				\draw (3,2)--(4,1)--(4,0);
				\node at (5.2,2.5){\, $\stackrel{\text{def}}{=}$};
				
				\draw (8.5,5)--(8.5,3);
				\node at (8.5,2.5){$\scriptstyle a_1$};
				\draw (8,2)--(7,1)--(7,0);
				\draw (9,2)--(10,1);
				\draw (11.5,5)--(11.5,3);
				\node at (11.5,2.5){$\scriptstyle a_2$};
				\draw (11,2)--(10,1)--(10,0);
				\draw (12,2)--(13,1);
				\draw (14,2)--(13,1)--(13,0);
				\node at (14.7,2.5){$\scriptstyle \cdots$};
				\draw (15,2)--(16,1);
				\draw (17.5,5)--(17.5,3);
				\node at (17.5,2.5){$\scriptstyle a_r$};
				\draw (17,2)--(16,1)--(16,0);
				\draw (18,2)--(19,1)--(19,0);
				\end{tikzpicture}}
		\end{center}
		Then, for any pair $k,n\geq1$ we have the following identity in $\MS$
		\begin{center}
			\boxed{\begin{tikzpicture}[scale=.25]								
				\draw (.5,6)--(1.5,5);
				\node at (2,4.5){$\scriptstyle k$};
				\draw (3.5,6)--(2.5,5);
				\draw (2,4)--(2,3);
				\node at (2,2.5){$\scriptstyle n$};
				\draw (1.5,2)--(.5,1);
				\draw (2.5,2)--(3.5,1);
				\node at (6,3.5){$=$};
				\end{tikzpicture}
				\begin{tikzpicture}[scale=.25]
				\draw (0,5)--(0,4)--(1,3);
				\draw (4,5)--(4,4)--(3,3);
				\node at (2,2.5){$\scriptstyle \langle a\rangle$};
				\draw (1,2)--(0,1)--(0,0); 
				\draw (3,2)--(4,1)--(4,0);
				
				\node at (-3,2.5){$\displaystyle{\sum_{a\in A(k,n)}^{\phantom{n+1}}}$};
				
				\end{tikzpicture}}			
		\end{center} 
	\end{lemma}
	We delay the proof of this lemma until after establishing Theorem \ref{Isomorphism MS and Sur} with its help.
	
	\begin{proof} [Proof of Theorem \ref{Isomorphism MS and Sur}]
		Because of Lemma \ref{splitting lemma}, the composition of two surjection-like elements satisfies
		\begin{center}
			\boxed
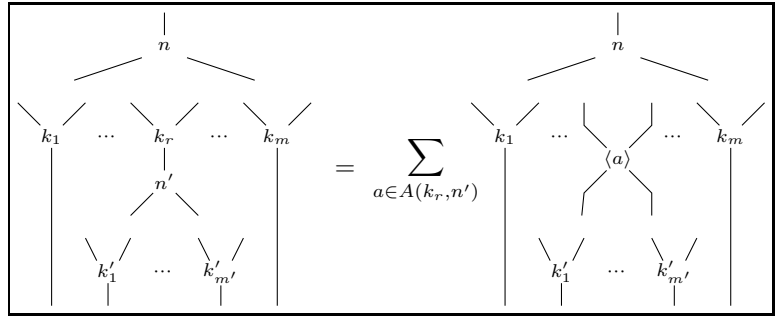
{\begin{tikzpicture}[scale=.3]
				\draw (6,14)--(6,13);
				\node at (6,12.5){$\scriptstyle n$};
				\draw (5,12)--(2,11);
				\draw (7,12)--(10,11);
				\node at (3.5,8.5){$\scriptstyle \cdots$};
				\node at (8.5,8.5){$\scriptstyle \cdots$};
				
				\draw (-.5,10)--(.5,9);
				\draw (2.5,10)--(1.5,9);
				\node at (1,8.5){$\scriptstyle k_1$};
				\draw (1,8)--(1,1);
				\draw (4.5,10)--(5.5,9);
				\draw (7.5,10)--(6.5,9);
				\node at (6,8.5){$\scriptstyle k_r$};
				\draw (9.5,10)--(10.5,9);
				\draw (12.5,10)--(11.5,9);
				\node at (11,8.5){$\scriptstyle k_m$};
				\draw (11,8)--(11,1);
				
				\draw (6,8)--(6,7);
				\node at (6,6.5){$\scriptstyle n'$};
				\draw (5.5,6)--(4.5,5);
				\draw (6.5,6)--(7.5,5);
				\draw (2.5,4)--(3,3);
				\draw (4.5,4)--(4,3);
				\node at (3.5,2.5){$\scriptstyle k'_1$};
				\draw (3.5,2)--(3.5,1);
				\draw (7.5,4)--(8,3);
				\draw (9.5,4)--(9,3);
				\node at (8.5,2.5){$\scriptstyle k_{m'}'$};
				\draw (8.5,1.9)--(8.5,1);
				\node at (6,2.5){$\scriptstyle \cdots$};
				
				\node at (14,7){$=$};
				\end{tikzpicture}
				\begin{tikzpicture}[scale=.3]
				\node at (-2.5,7){$\displaystyle{\sum_{a\in A(k_r,n')}}$};
				
				\draw (6,14)--(6,13);
				\node at (6,12.5){$\scriptstyle n$};
				\draw (5,12)--(2,11);
				\draw (7,12)--(10,11);
				\node at (3.5,8.5){$\scriptstyle \cdots$};
				\node at (8.5,8.5){$\scriptstyle \cdots$};
				
				\draw (-.5,10)--(.5,9);
				\draw (2.5,10)--(1.5,9);
				\node at (1,8.5){$\scriptstyle k_1$};
				\draw (1,8)--(1,1);
				\draw (4.5,10)--(4.5,9)--(5.5,8);
				\draw (7.5,10)--(7.5,9)--(6.5,8);
				\node at (6,7.5){$\scriptstyle \langle a \rangle$};
				\draw (9.5,10)--(10.5,9);
				\draw (12.5,10)--(11.5,9);
				\node at (11,8.5){$\scriptstyle k_m$};
				\draw (11,8)--(11,1);
				
				\draw (5.5,7)--(4.5,6)--(4.4,5);
				\draw (6.5,7)--(7.5,6)--(7.5,5);
				\draw (2.5,4)--(3,3);
				\draw (4.5,4)--(4,3);
				\node at (3.5,2.5){$\scriptstyle k'_1$};
				\draw (3.5,2)--(3.5,1);
				\draw (7.5,4)--(8,3);
				\draw (9.5,4)--(9,3);
				\node at (8.5,2.5){$\scriptstyle k_{m'}'$};
				\draw (8.5,1.9)--(8.5,1);
				\node at (6,2.5){$\scriptstyle \cdots$};
				\end{tikzpicture}}
		\end{center}
		From this, the proof of Theorem \ref{Isomorphism MS and Sur} follows using the fact that for $k,m\geq1$ the set $J(k,m)$ (used for operadic composition of surjections) can be identify with $A(k,m)$ through the map sending 
		\begin{equation*}
		1 = j_0 \leq \cdots \leq j_k = m
		\end{equation*} 
		to 
		\begin{equation*}
		\big \langle j_1-j_0\,,\, \dots\,,\, j_m-j_{m-1} \big \rangle.
		\end{equation*}
		Keeping this in mind, we can unwind the definitions and verify the required composition compatibility of the map $U(\MS)\to Sur$.
	\end{proof}
	
	\begin{proof}[Proof of Lemma \ref{splitting lemma}]
		For $n=k=2$ the statement is precisely the Leibniz relation. For $n=2$ and $k>2$ we use the following inductive argument. 
		\begin{center}
			\boxed
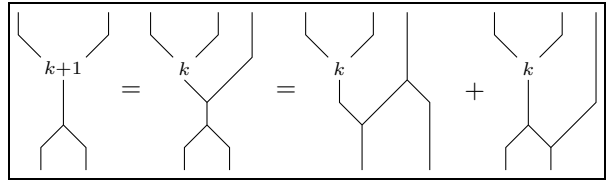
{\begin{tikzpicture}[scale=.3]
				\draw (0,7)--(0,6)--(1,5);
				\node at (2,4.5){$\scriptstyle k+1$};
				\draw  (2,4)--(2,2)--(1,1)--(1,0);
				\draw (4,7)--(4,6)--(3,5);
				\draw (2,2)--(3,1)--(3,0);
				\node at (5,3.5){=};
				\end{tikzpicture}
				\begin{tikzpicture}[scale=.3]
				\draw (0.5,7)--(0.5,6)--(1.5,5);
				\node at (2,4.5){$\scriptstyle k$};
				\draw  (2,4)--(3,3)--(3,2)--(2,1)--(2,0);
				\draw (3.5,7)--(3.5,6)--(2.5,5);
				\draw (3,2)--(4,1)--(4,0);
				\draw (5,7)--(5,5)--(3,3);
				\node at (6.5,3.5){=};
				\end{tikzpicture}
				\begin{tikzpicture}[scale=.3]
				\draw (0.5,7)--(0.5,6)--(1.5,5);
				\node at (2,4.5){$\scriptstyle k$};
				\draw  (2,4)--(2,3)--(3,2)--(3,0);
				\draw (3.5,7)--(3.5,6)--(2.5,5);
				\draw (5,7)--(5,4)--(3,2);
				\draw (5,4)--(6,3)--(6,0);
				\node at (8,3.5){+};
				\end{tikzpicture}
				\begin{tikzpicture}[scale=.3]
				\draw (0.5,7)--(0.5,6)--(1.5,5);
				\node at (2,4.5){$\scriptstyle k$};
				\draw  (2,4)--(2,2)--(1,1)--(1,0);
				\draw (3.5,7)--(3.5,6)--(2.5,5);
				\draw (2,2)--(3,1)--(3,0);
				\draw (5,7)--(5,3)--(3,1);
				\end{tikzpicture}}
		\end{center}
		which by induction equals	
		\begin{center}
			\boxed
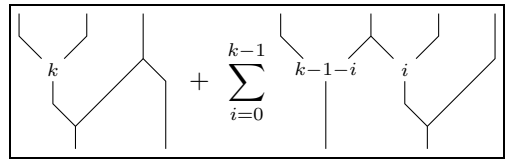
{\begin{tikzpicture}[scale=.3]
				\draw (0.5,6)--(0.5,5)--(1.5,4);
				\draw (3.5,6)--(3.5,5)--(2.5,4);
				\node at (2,3.5){$\scriptstyle k$};
				\draw (2,3)--(2,2)--(3,1)--(3,0);
				\draw (6,6)--(6,4)--(7,3)--(7,0);
				\draw (6,4)--(3,1); 
				\node at (8.5,3){+};
				\end{tikzpicture}
				\begin{tikzpicture}[scale=.3]
				\node at (-1.5,3){$\displaystyle{\sum_{i=0}^{k-1}}$};
				\draw (0,6)--(0,5)--(1,4);
				\node at (2,3.5){$\scriptstyle k-1-i$};
				\draw (2,3)--(2,0);
				\draw (4,6)--(4,5)--(3,4);
				\draw (4,5)--(5,4);
				\node at (5.5,3.5){$\scriptstyle i$};
				\draw (5.5,3)--(5.5,2)--(6.5,1)--(6.5,0);
				\draw (7,6)--(7,5)--(6,4);
				\draw (9.5,6)--(9.5,4)--(6.5,1);
				\end{tikzpicture}}
		\end{center}
		which in turns equals
		\begin{center}
			\boxed
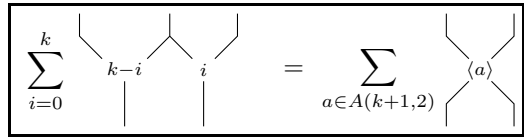
{\begin{tikzpicture}[scale=.3]
				\node at (-1.5,2.5){$\displaystyle{\sum_{i=0}^{k}}$};
				\draw (0,5)--(0,4)--(1,3);
				\node at (2,2.5){$\scriptstyle k-i$};
				\draw (2,2)--(2,0);
				\draw (4,5)--(4,4)--(3,3);
				\draw (4,4)--(5,3);
				\node at (5.5,2.5){$\scriptstyle i$};
				\draw (5.5,2)--(5.5,0);
				\draw (7,5)--(7,4)--(6,3);
				\node at (9.5,2.5){=};
				\end{tikzpicture}
				\begin{tikzpicture}[scale=.3]
				\node at (-3,2.5){$\displaystyle{\sum^{\phantom{n+1}}_{a\in A(k+1,2)}}$};
				\draw (0,5)--(0,4)--(1,3);
				\draw (3,5)--(3,4)--(2,3);
				\node at (1.5,2.5){$\scriptstyle \langle a \rangle$};
				\draw (1,2)--(0,1)--(0,0);
				\draw (2,2)--(3,1)--(3,0);
				\end{tikzpicture}}
		\end{center}
		as desired. Now, we proceed by induction on $n$.
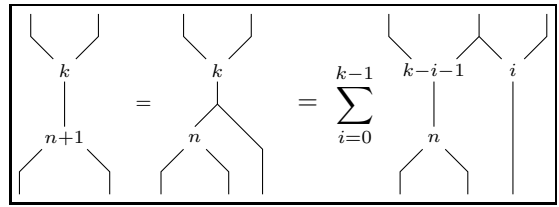
		\begin{center}
			\boxed{\begin{tikzpicture}[scale=.3]
				\draw (0.5,8)--(0.5,7)--(1.5,6);
				\draw (3.5,8)--(3.5,7)--(2.5,6);
				\node at (2,5.5){$\scriptstyle k$};
				\draw (2,5)--(2,3);
				\node at (2,2.5){$\scriptstyle n+1$};
				\draw (1,2)--(0,1)--(0,0);
				\draw (3,2)--(4,1)--(4,0);
				\node at (5.5,4){$\scriptstyle =$};
				\end{tikzpicture}
				\begin{tikzpicture}[scale=.3]
				\draw (0.5,8)--(0.5,7)--(1.5,6);
				\draw (3.5,8)--(3.5,7)--(2.5,6);
				\node at (2,5.5){$\scriptstyle k$};
				\draw (2,5)--(2,4)--(4,2)--(4,0);
				\draw (2,4)--(1,3);
				\node at (1,2.5){$\scriptstyle n$};
				\draw (0.5,2)--(-.5,1)--(-.5,0);
				\draw (1.5,2)--(2.5,1)--(2.5,0);
				\node at (6,4){$=$};
				\end{tikzpicture}
				\begin{tikzpicture}[scale=.3]
				\node at (-1.5,4){$\displaystyle{\sum_{i=0}^{k-1}}$};
				\draw (0,8)--(0,7)--(1,6);
				\draw (4,8)--(4,7)--(3,6);
				\node at (2,5.5){$\scriptstyle k-i-1$};
				\draw (2,5)--(2,3);
				\node at (2,2.5){$\scriptstyle n$};
				\draw (1.5,2)--(0.5,1)--(0.5,0);
				\draw (2.5,2)--(3.5,1)--(3.5,0);
				\draw (4,7)--(5,6);
				\draw (7,8)--(7,7)--(6,6);
				\node at (5.5,5.5){$\scriptstyle i$};
				\draw (5.5,5)--(5.5,0);
				\end{tikzpicture}}
		\end{center}
		where we used the case $n=2$ proven before. Now, by induction, the sum above equals
		\begin{center}
			\boxed{	\begin{tikzpicture}[scale=.3]
				\node at (-4,3.5){$\displaystyle{\sum_{i=0}^{k-1}\ \sum_{a\in A(k-i-1,n)}^{\phantom{k+1}}}$};
				\draw (0.5,6)--(0.5,5)--(1.5,4);
				\draw (3.5,6)--(3.5,5)--(2.5,4);
				\node at (2,3.5){$\scriptstyle \langle a \rangle$};
				\draw (3.5,5)--(4.5,4);
				\draw (6.5,6)--(6.5,5)--(5.5,4);
				\node at (5,3.5){$\scriptstyle i$};
				\draw (5,3)--(5,1);
				\draw (1.5,3)--(0.5,2)--(0.5,1);
				\draw (2.5,3)--(3.5,2)--(3.5,1);
				\node at (9,3.5){$=$};
				\end{tikzpicture}
				\begin{tikzpicture}[scale=.3]
				\node at (-3.5,3.5){$\displaystyle{\sum_{b\in A(k,n+1)}^{\phantom{k+1}}}$};
				\draw (0,6)--(0,5)--(1,4);
				\draw (3,6)--(3,5)--(2,4);
				\node at (1.5,3.4){$\scriptstyle \langle b \rangle$};
				\draw (1,3)--(0,2)--(0,1);
				\draw (2,3)--(3,2)--(3,1);
				\end{tikzpicture}}
		\end{center}
		where the assignment takes $a\in A(k-i-1,n)$ to $b\in A(k,n+1)$ with
		\begin{displaymath}
		b_j = 
		\begin{cases}
		\ \ \ a_j & \text{if } j < k-i-1 \\
		a_{k-i}+1 & \text{if } j = k-i-1 \\
		\ \ \ \; 0 & \text{if } j > k-i-1.
		\end{cases}
		\end{displaymath}
		Therefore, we concludes the proof of this lemma and of Theorem \ref{Isomorphism MS and Sur}.
	\end{proof}
	
	\subsection{Comparing $E_\infty$-coalgebra structures}
	
	Next we describe the natural $Sur$-coalgebra structure constructed by McClure-Smith and Berger-Fresse on the normalized chains of simplicial sets. We present it with $\F_2$-coefficients and refer to \cite{mcclure2003multivariable} and \cite{berger2004combinatorial} for their respective sign conventions. 
	
	As explained in Corollary \ref{simplicial set representations}, such family of structures is equivalently described by natural operad morphisms
	\begin{equation*}
	\phi[d]:Sur \to \End(\chains(\Delta^d)).
	\end{equation*}
	
	\begin{definition} \label{coaction of Sur}
		For every object $d$ in the simplex category and non-negative integer $m$. The map 
		$$\phi[d](m):Sur(m)\tensor\chains(\Delta^d)\to\chains(\Delta^d)^{\tensor m}$$
		is defined on basis elements as follows. For any $\sigma=[v_0,\dots,v_d]\in\chains(\Delta^d)$ and $s=(s(1),\dots,s(n))\in Sur(m)$ the element $\phi[d](m)\big(s\tensor[v_0,\dots,v_d]\big)$ is of the form 
		\begin{equation}\label{sum after iterated coproduct}
		\sum_I\sigma^1_I\tensor\cdots\tensor\sigma^m_I
		\end{equation}
		where the sum ranges over all sequences $I=\langle i_0,\dots,i_n\rangle$ such that $0=i_0\leq i_1\leq\cdots\leq i_{n-1}\leq i_n=d$, and each $\sigma^r_I$ is a face of $\sigma$ that we now define. For a fix sequence $I$ and $1\leq r\leq m$, let $\{j_1^r,\dots,j^r_k\}=s^{-1}(r)$; define $\sigma_I^r$ to be equal to either 
		\begin{equation}\label{after join}
		[v_{i_{j^r_{1}}-1},\dots,v_{i_{j^r_{1}}},v_{i_{j^r_{2}}-1},\dots,v_{i_{j^r_{2}}},\,\dots\,v_{i_{j^r_{k}}-1},\dots,v_{i_{j^r_{k}}}]
		\end{equation} if all vertices are distinct or to be $0$ if they are not.
	\end{definition}
	
	\begin{remark}
		The operad morphism $U(\varphi) : U(\S) \to \End(\chains(X))$ does not factor through $U(\MS)$. Examples can be constructed using the following observation: the image in $\End(\chains(\Delta^2))$ of the Leibniz element does not map $[0,2]\tensor[1]$ to $0$. This morphism, as we will show, does factor when restricted to the suboperad generated by surjection-like elements, a suboperad that maps surjectively onto $\MS$.
	\end{remark}
	
	\begin{theorem} \label{Representations of MS and Sur on simplicial sets}
		Let $U(\Sh)$ be the suboperad of $U(\S)$ generated by surjection-like elements. Then, for every simplicial set $X$, the following diagram commutes:
		\begin{center}
			\begin{tikzcd}
			U(\S)  \arrow[r, "U(\varphi)"] & \End(\chains(X)) \\
			U(\Sh) \arrow[u, hook] \arrow[d, two heads] &        \\
			U(\MS) \arrow[r, "\cong"] & Sur \arrow[uu, "\phi"']
			\end{tikzcd}
		\end{center}
	\end{theorem}
	
	\begin{proof}
		We start by verifying that $U(\varphi) : U(\S) \to \End(\chains(X))$ factor through $U(\MS)$ when restricted to $U(\Sh)$. The fact that the join map is associative with $\F_2$-coefficients follows from the associativity of union of sets. The Alexander-Whitney map being coassociative is a classical result that can be easily verified. To see that the involution element is sent to the 0 map, we recall that the Alexander-Whitney map of any simplex is a sum of tensor pairs of simplices that share a vertex, and that the join of two simplices that contain a common vertex equals $0$. To verify that the Leibniz element is sent to the map that takes any $[v_0,\dots,v_p]\tensor[w_0,\dots,w_q]$ to $0$, we notice that the order preserving condition imposed surjection-like elements allows us to assume $v_i\leq w_j$ for all $i$ and $j$. Therefore, the images of both
		\begin{center}
			\boxed{\begin{tikzpicture}[scale=.2]
				\draw (0,4)--(1,3)--(1,1)--(0,0);
				\draw (2,4)--(1,3);
				\draw (1,1)--(2,0);
				\node at (6,2){ \& };
				\end{tikzpicture}
				\ \ \ 
				\begin{tikzpicture}[scale=.2]
				\draw (5,4)--(5,3)--(4,2)--(4,0);
				\draw (5,3)--(7,1)--(7,0);
				\draw (8,4)--(8,2)--(7,1)--(7,0);
				\node at (9,2){$\scriptstyle +$};
				
				\draw (10,4)--(10,2)--(11,1)--(11,0);
				\draw (13,4)--(13,3)--(11,1);
				\draw (13,3)--(14,2)--(14,0);
				
				\node at (10,2){};
				\end{tikzpicture}}
		\end{center}
		send $[v_0,\dots,v_p]\tensor[w_0,\dots,w_q]$ to  
		\begin{equation*}
		\sum_{i=0}^p[v_0,\dots,v_i]\tensor[v_i,\dots,v_p,w_0,\dots,w_q]+\sum_{j=0}^q[v_0,,\dots,v_p,w_0,\dots,w_j]\tensor[w_j,\dots,w_q],
		\end{equation*}
		which concludes the verification.
		
		We now show that the diagram is commutative. Let $s\in Sur(m)$ be a surjection of degree $(n-m)$ and for every $r = 1,\dots,m$ let 
		\begin{equation*}
		s^{-1}(r) = \{i^r_1,\dots,i^r_{k_r}\}.
		\end{equation*}
		Then, the surjection-like element whose associated surjection is $s$ factors as follows:
		\begin{center}
			\boxed{\begin{tikzpicture}[scale=.8]
				\draw (0,0)--(0,-.6) node[below]{$\scriptstyle 1$};
				\draw (0,0)--(.5,.5);
				\draw (0,0)-- (-.2,.5) node[above]{\ \ \ $\scriptstyle1\ 2\  ...\  k_1$};
				\draw (-.5,.5)--(0,0);
				
				\node at (1,0){$\scriptstyle\cdots$};
				\node at (1,-.9){$\scriptstyle\cdots$};
				
				\draw (2,0)--(2,-.68) node[below]{$\scriptstyle m$};
				\draw (2,0)--(2.5,.5);
				\draw (2,0)-- (1.8,.5) node[above]{\ \ \ $\scriptstyle 1\ 2\  ...\  k_m$};
				\draw (1.5,.5)--(2,0);
				
				\draw (1,2.5)--(1,3);
				\draw (1,2.5)--(0,2) node[below]{$\scriptstyle 1$};
				\draw (1,2.5)--(.5,2) node[below]{$\scriptstyle 2$};
				\draw (1,2.5)--(2,1.95) node[below]{$\scriptstyle n$};
				\node at (1.2,1.75){$\scriptstyle ...$};
				\node at (3.8,1){=};
				\end{tikzpicture}
				\begin{tikzpicture}[scale=1.2]
				\draw (0,1)--(0,.4) node[below]{$\scriptstyle 1$};
				\draw (0,1)--(.5,1.5);
				\draw (0,1)-- (-.2,1.5) node[above]{\ \ \ $\scriptstyle i_1^1\ i_2^1\ \ ...\ \ i_{k_1}^1$};
				\draw (-.5,1.5)--(0,1);

				\node at (1,1){$\scriptstyle\cdots$};
				\node at (1,.2){$\scriptstyle\cdots$};
				
				\draw (2,1)--(2,.32) node[below]{$\scriptstyle m$};
				\draw (2,1)--(2.5,1.5);
				\draw (2,1)-- (1.8,1.5) node[above]{\ \ \ \ $\scriptstyle i_1^m\ i_2^m\ \, ...\ i_{k_m}^m$};
				\draw (1.5,1.5)--(2,1);
				
				\node at (2,-.25){};
				\node at (3.3,1){$\circ$};
				
				\draw (5,1)--(5,1.4);
				\draw (5,1)--(4,.6) node[below]{$\scriptstyle 1$};
				\draw (5,1)--(4.5,.6) node[below]{$\scriptstyle 2$};
				\draw (5,1)--(6,.6) node[below]{$\scriptstyle n$};
				\node at (5.2,.45){$\scriptstyle ...$};
				\end{tikzpicture}}
		\end{center}		 
		Recall from Theorem \ref{chain representations} that \par
		\begin{center}
			\boxed{\begin{tikzpicture}[scale=.8]
				\draw (1,2.5)--(1,3);
				\draw (1,2.5)--(0,2) node[below]{$\scriptstyle 1$};
				\draw (1,2.5)--(.5,2) node[below]{$\scriptstyle 2$};
				\draw (1,2.5)--(2,1.95) node[below]{$\scriptstyle n$};
				\node at (1.2,1.75){$\scriptstyle ...$};
				\node at (2.9,2.35){$[v_0,\dots,v_d]$};
				\node at (9.7,2.2){$\displaystyle =\  \sum_I\,[v_{i_0},\dots,v_{i_1}]\tensor[v_{i_1},\dots,v_{i_2}]\tensor\cdots\tensor[v_{i_{n-1}},\dots,v_{i_n}]$};
				\end{tikzpicture}}
		\end{center} 
		where the sum ranges over all sequences $0=i_0\leq i_1\leq\cdots\leq i_{n-1}\leq i_n=d$. 
		
		If we then apply \par
		\begin{center}
			\boxed{\begin{tikzpicture}[scale=1]
				\draw (0,1)--(0,.4) node[below]{$\scriptstyle 1$};
				\draw (0,1)--(.5,1.5);
				\draw (0,1)-- (-.2,1.5) node[above]{\ \ \ $\scriptstyle i_1^1\ i_2^1\ \ ...\ \ i_{k_1}^1$};
				\draw (-.5,1.5)--(0,1);

				\node at (1,1){$\scriptstyle\cdots$};
				\node at (1,.2){$\scriptstyle\cdots$};
				
				\draw (2,1)--(2,.32) node[below]{$\scriptstyle m$};
				\draw (2,1)--(2.5,1.5);
				\draw (2,1)-- (1.8,1.5) node[above]{\ \ \ \ $\scriptstyle i_1^m\ i_2^m\ \, ...\ i_{k_m}^m$};
				\draw (1.5,1.5)--(2,1);
				\end{tikzpicture}}
		\end{center}
		we obtain
		\begin{equation*}
		\sum_I\sigma^1_I\tensor\cdots\tensor\sigma^m_I
		\end{equation*}
		with $\sigma_I^r$ agreeing with (\ref{after join}) in Definition \ref{coaction of Sur}. This concludes the proof.
	\end{proof}
	
	\section{The augmented simplex category and the prop $\S$}
	In this appendix we describe a family, natural with respect to augmented simplicial maps, of opposite $\S$-bialgebra structures on the chains of the standard augmented simplices, i.e., natural prop morphisms $$\big\{\S\to\End^{\mathrm{op}}(\chains(\Delta_+^d))\big\}_{d\in\Delta_+\,.}$$	
	The image of the three generators of $\S$ are the join map, the empty simplex, and the Alexander-Whitney map. Comparing with the family of $\S$-coalgebra structures on the chains of the standard simplices described in Section 4, we see a duality between diagonal approximations and join operations in the usual and augmented simplicial contexts.
	
	\begin{notation}
	    Consider the augmented simplex category $\Delta_+$ whose objects are the sets $$\bar 0=\emptyset,\ \bar 1=\{1\},\ \bar 2=\{1,2\},\ \dots$$ and whose morphisms are order preserving maps. The normalized chain complex functor is defined analogously to the non-augmented case. It satisfies $\partial [v]=[\emptyset]$ for all vertices and it is graded by cardinality instead of dimension. The functor of normalized cochains is defined by postcomposing with the functor of linear duality. We respectively denote the images of $\Hom_{\Delta_+}(-,\bar{d})$ by $\chains^a(\Delta_+^d)$ and $\cochains_a(\Delta_+^d)$.
	\end{notation}	
	
	\begin{theorem}\label{augmented simplicial sets representation}
		For every object $\bar{d}$ of the augmented simplex category, the following assignment of generators defines an opposite $\S$-bialgebra $\S\to\End(\chains^a(\Delta^d_+))$ natural with respect to augmented simplicial maps: \par
		\ \ \ \ $\counit\in\Hom(k,\chains^a(\Delta^d_+))_0$ is defined by
		$$\counit\,(1)=[\emptyset].$$
		\ \ \ $\coproduct\in\Hom(\chains^a(\Delta^d_+)^{\otimes2},\chains^a(\Delta^d_+))_0$ is defined by
		$$\coproduct\ \big(\left[v_0,\dots,v_p \right] \otimes\left[v_{p+1},\dots,v_q\right]\big)=\begin{cases}\text{sign}(\pi)\cdot\left[v_{\pi(0)},\dots,v_{\pi(q)}\right] & \text{ if } i\neq j \text{ implies } v_i\neq v_j \\ \hspace*{1.15cm}
		0 & \text{ if not.} \end{cases}$$
		\ \ \ where $\pi$ is the permutation that orders the totally ordered set of vertices.\par
		\ \ \ $\product\in\Hom(\chains^a(\Delta^d_+),\chains^a(\Delta^d_+)^{\otimes2})_1$ is defined by
		$$\product\ [v_0,\dots,v_q]=\sum_{i=0}^q(-1)^i[v_0,\dots,v_i]\otimes[v_i,\dots,v_q].$$ 
		\begin{proof}
			Throughout this proof we identify\ \counit\ , \coproduct \ and \product \ with their images in $\End(\chains^a(\Delta^d_+))$. We need to verify that the relations of $\S$ are satisfied by these images.
			
			The fact that the maps \leftcounitality\ and\ \rightcounitality\ are equal to the zero map follows from the fact that the empty set is the unit for the union of sets.
			The formula defining \product\ as well as the fact that it has degree $1$ makes the definition of\ \productcounit\ as the zero map the only sensible one.  
			In order to establish $\partial\ \product=\ \boundary$\;, we must consider two cases for the basis element to which the map is applied:
			\begin{enumerate}[leftmargin=.6cm]
				\item Degree $0$ \par
				\begin{center}
					\begin{tikzcd}
						& {[\emptyset]} \arrow[swap, dl, maps to, "\product", out=180, in=90] \arrow[d, maps to, "\boundary\ ", swap] \arrow[dr, maps to, "\displaystyle{\partial}", in=90, out=0]  &  \\
						{0} \arrow[dr, |- , "\displaystyle{\partial}", out=-90, in=180, swap] & 
						{0}  & 
						{0}  \arrow[swap, dl, |- , "\product",out=-90, in=0, swap]\\
						& + \uar &
					\end{tikzcd}
				\end{center}
				
				\item Degree greater than $0$\par 		
				\begin{center}
					\begin{tikzcd}[column sep=-25]
						& {[v_0,\dots,v_p]} \arrow[swap, ddl, maps to, "\product", out=180, in=90] \arrow[d, maps to, "\boundary\ ", swap] \arrow[ddr, maps to, "\displaystyle{\partial}", in=90, out=0]  &	[-10pt]  \\
						& {[\emptyset]\tensor[v_0,\dots,v_p]-[v_0,\dots,v_p]\tensor[\emptyset]}  & \\
						{\sum_i (-1)^i [v_0,\dots,v_i]\tensor[v_i\dots,v_p]} \arrow[r, |- , "\displaystyle{\partial}", out=-90, in=225]
						& + \uar &
						{\sum_{i}(-1)^{i}[v_0,\dots,\hat{v}_i,\dots,v_p]}  \arrow[l, |- , "\product", out=-90, in=-45,swap]  
					\end{tikzcd}
				\end{center}
			\end{enumerate}
			
			\item In order to establish $\partial\ \coproduct=0$, we must consider four cases for the basis elements to which the map is applied: 
			\begin{enumerate}[leftmargin=.6cm]
				\item Both of degree $0$ \par
				\begin{center}
					\begin{tikzcd}
						& {[\emptyset]\tensor[\emptyset]} \arrow[swap, dl, maps to, "\coproduct", out=180, in=90] \arrow[d, maps to] \arrow[dr, maps to, "\displaystyle{\partial}", in=90, out=0]  &  \\
						{0} \arrow[dr, |- , "\displaystyle{\partial}", out=-90, in=180, swap] & 
						{0}  & 
						{0}  \arrow[swap, dl, |- , "\coproduct",out=-90, in=0, swap]\\
						& - \uar &
					\end{tikzcd}
				\end{center}
				
				\item Only one of degree $0$\par 		
				\begin{center}
					\begin{tikzcd}[column sep=-15pt]
						& {[\emptyset]\tensor[v_0,\dots,v_p]} \arrow[swap, dl, maps to, "\coproduct", out=180, in=90] \arrow[d, maps to] \arrow[dr, maps to, "\displaystyle{\partial}", in=90, out=0]  &	[-10pt]  \\
						{[v_0,\dots,v_p]} \arrow[dr, |- , "\displaystyle{\partial}", out=-90, in=190,swap] &
						{0}  &
						{\sum_{i}(-1)^{i}[\emptyset]\tensor[v_0,\dots,\hat{v}_i,\dots,v_p]}  \arrow[dl, |- , "\coproduct", out=-90, in=-15, swap]  \\
						& - \uar &
					\end{tikzcd}
				\end{center}
				
				\item None of degree $0$ and sharing a vertex\par
				\begin{center}
					\begin{tikzcd}[column sep=tiny]
						& 	{[v_0,\dots,v_p]\tensor[w_0,\dots,w_q]} \arrow[swap, dl, maps to, "\coproduct",out=180, in=90] \arrow[swap, d, maps to ] \arrow[dr, maps to, "\displaystyle{\partial}",in=90, out=0]  & 	[-45pt] \\
						{0} \arrow[dr, |- , "\displaystyle{\partial}", out=-90, in=195, swap]& 
						{0}  &
						\begingroup	\renewcommand*{\arraystretch}{1.5}
						\begin{matrix}\sum_{i}(-1)^{i}[v_0,\dots,\hat v_i,\dots,v_p]\tensor[w_0,\dots,w_q]\ + \\(-1)^{p+1}\sum_{j}(-1)^{j}[v_0,\dots,v_p]\tensor[w_0,\dots,\hat{w}_j,\dots,w_q]	\end{matrix}
						\endgroup 	\arrow[dl, |- , "\coproduct", out=-90, in=-15, swap] \\
						& - \uar &
					\end{tikzcd}
				\end{center}
				
				\item None of degree $0$ and not sharing a vertex\par
				\begin{center}
					\begin{tikzcd}[column sep=-45pt]
						& {[v_0,\dots,v_p]\tensor[v_{p+1},\dots,v_q]} \arrow[swap, dl, maps to, "\coproduct", out=180, in=90] \arrow[d, maps to] \arrow[dr, maps to, "\displaystyle{\partial}", in=90, out=0]  & [-0pt]  \\
						{\begin{matrix}	sign(\pi)\ \cdot [v_{\pi(0)},\dots,v_{\pi(q)}]\end{matrix}}  \arrow[dr, |- , "\displaystyle{\partial}", out=-90, in=195, swap] & 
						{0}  & 
						\begingroup
						\renewcommand*{\arraystretch}{1.5}
						\begin{matrix}
							\sum_{i}(-1)^{i}[v_0,\dots,\hat v_i,\dots,v_p]\tensor[v_{p+1},\dots,v_q]\ + \\ 
							\sum_{i}(-1)^{p+1+i}[v_0,\dots,v_p]\tensor[v_{p+1},\dots,\hat{v}_i,\dots,v_q]
						\end{matrix}
						\endgroup \arrow[dl, |- , "\coproduct", out=-90, in=-15, swap] \\
						& - \uar &
					\end{tikzcd}
				\end{center}
			\end{enumerate}
		\end{proof}
		
	\end{theorem}
	
	\begin{remark}
		Adamaszek and Jones have also explored a relationship between the higher joins and the Steenrod diagonal in \cite{adamaszek2013symmetric}.
	\end{remark}
	
	\bibliography{medina2018algebraic}{}
	\bibliographystyle{alpha}
	
\end{document}